\documentclass[11pt,a4paper,twoside]{article} 

\usepackage{amsmath,amssymb,graphicx,xcolor}
\usepackage{mathtools,amsthm}
\mathtoolsset{showonlyrefs} %hide unused labels

\usepackage[cal=pxtx]{mathalpha} %new calligraphic font for V_ovlap to be nice
\usepackage[utf8]{inputenc}
\usepackage[T1]{fontenc}
\usepackage{framed,xcolor}
\usepackage{tikz,pgfplots}\pgfplotsset{compat=1.9}
\usetikzlibrary{pgfplots.groupplots}
\usepackage[dvipsnames]{xcolor}
\usepackage[pagebackref=true,colorlinks=true, linkcolor=RoyalPurple,citecolor=RoyalPurple,urlcolor=RoyalPurple]{hyperref}
\renewcommand*\backref[1]{\ifx#1\relax \else (Cited on #1) \fi}

\usepackage{enumitem}
\setlist[itemize]{label={\footnotesize\textcolor{gray}{\textbullet}},leftmargin=15pt}
\usepackage[margin=2.5cm]{geometry}
\usepackage{dsfont} % pour le 1 double barre
\usepackage{aurical} % for the unit ball volume
\usepackage[english]{babel}
\theoremstyle{definition}
\newtheorem{definition}{Definition}[section]
\newtheorem{theorem}[definition]{Theorem}
\newtheorem{proposition}[definition]{Proposition}
\newtheorem{lemma}[definition]{Lemma}

\theoremstyle{definition}
\newtheorem{remark}[definition]{Remark}

\numberwithin{equation}{section}

\definecolor{shadecolor}{gray}{0.975} % 0.975
\long\def\DETAILED#1\endDETAILED{\begin{shaded} #1\end{shaded}} % the default version is the long one
\long\def\DETAILED#1\endDETAILED{} %  Cancel % at the beginning of this line to switch on the short 

\DeclarePairedDelimiter\floor{\lfloor}{\rfloor}

\usepackage[only,varodot]{stmaryrd}
\def \rdep {{\overset{\varodot}{r}}}

\def \disp {\displaystyle}
\def \eps {{\varepsilon}}   \def \phi {{\varphi}}
\newcommand{\un}{\mathds{1}}
\def \N {{\mathbb N}}    \def \R {{\mathbb R}}   \def \C {\mathcal{C}}
\def \D {\mathcal{D}}
\newcommand*{\Gcal}{\mathcal{G}}
\def \X {{\mathbb X}} 
\def \M {\mathcal{M}}

\def \rgrande {\mathring{r}}
\newcommand*{\ngrande}{{n}}   \newcommand*{\npiccola}{{m}}
\newcommand*{\grande}{{\mathring{x}}}   \newcommand*{\piccola}{{\dot{x}}}
\newcommand*{\grandey}{{\mathring{y}}}   \newcommand*{\piccolay}{{\dot{y}}}
\newcommand*{\grandeX}{{\mathring{X}}}   \newcommand*{\piccolaX}{{\dot{X}}}
\newcommand*{\bx}{{\bf x}}    \newcommand*{\by}{{\bf y}} 
\newcommand*{\bgrande}{{\mathring{\bx}}}   \newcommand*{\bpiccola}{{\dot{\bx}}} 
\newcommand*{\WGrande}{{\mathring{W}}}
\newcommand*{\WPiccola}{{\dot{W}}}
\newcommand*{\sigmaP}{\dot{\sigma}}
\newcommand*{\Igrande}{{\mathring{\mathcal{J}}}}
\newcommand*{\Ipiccola}{{\dot{\mathcal{J}}}}
\newcommand*{\zpiccola}{{\dot{z}}}
\newcommand*{\zgrande}{{\mathring{z}}}
\newcommand*{\psiG}{{\mathring{\psi}}} \newcommand*{\psiGyR}{{\psiG^\by_R}}
\newcommand*{\psiP}{{\dot{\psi}}} \newcommand*{\psiPyR}{{\psiP^\by_R}}
\newcommand*{\BallGyR}{{\mathring{B}^\by_R}} \newcommand*{\BallPyR}{{\dot{B}^\by_R}}
\newcommand*{\lambdaGyR}{{\mathring{\lambda}}^\by_R}
\newcommand*{\lambdaPyR}{{\dot{\lambda}}^\by_R}

\newcommand*{\Zmixing}{{\mathcal{Z}}}
\newcommand*{\Zgrande}{{\mathring{Z}}}

\newcommand*{\volball}% for the unit ball volume
{\mathsf v}

\newcommand*{\Cgrande}{{\mathring{C}_d}}   \newcommand*{\Cpiccola}{{\dot{C}_d}}

\newcommand{\const}{C^{\rm st}}

\def \bX {{\bf X}}

\newcommand \nrg {{\mathcal{E}}}

\newcommand{\abs}[1]{\left\lvert #1\right\rvert}
\newcommand{\vol}[1]{\abs{#1}}
\newcommand{\bgrandey}{{\mathring{\by}}}   \newcommand{\bpiccolay}{{\dot{\by}}}

\newcommand*{\MGrande}{{\mathring{\M}}}
\newcommand*{\MPiccola}{{\dot{\M}}}

\newcommand{\BGrande}{{\mathbb{B}}} % union of the n balls with radii \rdep around the large particles
\def \rpiccola {\dot{r}} 
\def \rdep {{\overset{\varodot}{r}}}

\newcommand*{\Grande}{\mathring{X}}
\newcommand*{\Piccola}{\dot{X}}

\newcommand*{\phidep}{{\phi}}

\newcommand*{\mugrande}{{\mathring{\mu}}}

\newcommand{\BadPath}{{\mathcal{B}}}

\newcommand*{\Vovlap}{\mathcal{V}_{\rm{ovlap}}}
\newcommand*{\rspace}{\,\hspace{-.14em}}

\definecolor{ForestGreen}{RGB}{18,90,18} % Myriam

\definecolor{amethyst}{rgb}{0.6, 0.4, 0.8} % Alexander  

\author{Myriam Fradon\thanks{Univ. Lille, CNRS UMR 8524, Laboratoire P. Painlevé, {\tt myriam.fradon@univ-lille.fr}} \and Alexander Zass\thanks{Weierstrass Institute for Applied Analysis and Stochastics, Berlin, {\tt zass@wias-berlin.de}} }
\date{}
\title{Infinite-dimensional diffusions and depletion interaction for a model of colloids
}
\begin{document}%%%%%%%%%%%%%%%%%%%%%%%%%%%%%%%%%%%%%%%%%%%%%%%%%%%%%%%%%%%%%%%%%
%%%%%%%%%%%%%%%%%%%%%%%%%%%%%%%%%%%%%%%%%%%%%%%%%%%%%%%%%%%%%%%%%
\maketitle

\begin{abstract}
    We consider infinite-dimensional random diffusion dynamics for the Asakura--Oosawa model of interacting hard spheres of two different sizes. We construct a solution to the corresponding SDE with collision local times, analyse its reversible measures, and observe the emergence of an attractive short-range depletion interaction between the large spheres. We study the Gibbs measures associated to this new interaction, exploring connections to percolation and optimal packing. 
    \medbreak
\end{abstract}

%\tableofcontents
%>>>>>>>>>>>>>>>>>>>>>>>>>>>>>>>>>>>>>>>>>>>>>>>>>>>>>>>>>>>>>>>>>>>>>>>>>>>>>>>>>
\section{Introduction}
%>>>>>>>>>>>>>>>>>>>>>>>>>>>>>>>>>>>>>>>>>>>>>>>>>>>>>>>>>>>>>>>>>>>>>>>>>>>>>>>>>
In this work, we consider colloidal dynamics, that is the motion of large colloidal particles within a medium formed by much smaller solvent particles.
The model is inspired by the Asakura--Oosawa (AO) model~\cite{AO54} of Chemical Physics; there, the authors introduce a size-asymmetric binary mixture of hard spheres to model colloidal suspensions, successfully explaining the origin of an attractive interaction that is observed when looking at large colloidal particles in a solution of non-adsorbing polymers, and which is referred to as ``attraction through repulsion''~\cite{LT11}. We consider a modification of the model, due to A. Vrij~\cite{Vr76}, in which the colloids are represented as large hard spheres, and the polymers as small penetrable spheres, which we will call particles. We present here a two-fold study of this model -- both from the dynamic (infinite-dimensional SDEs) point of view, and from the static (infinite-volume Gibbs point processes) one.

Indeed, already in the early 20th century, J. Perrin observed that these objects are characterised by a Brownian-type motion~\cite{Perrin1909}. To preserve shift invariance, we are therefore modelling infinitely many hard spheres in Euclidean space, randomly diffusing through a random medium of infinitely many randomly diffusing very small particles.
Since spheres cannot overlap, and spheres and particles do not overlap either, the model is an infinite-dimensional SDE with normal reflection on the boundary of the set of forbidden, overlapping configurations.

\medbreak

Since the pioneering work of A.\,V. Skorokhod~\cite{Skorokhod1,Skorokhod2} on half-spaces, pathwise solutions of reflected stochastic differential equations have been constructed on finite-dimensional domains under boundary conditions of increasing generality: see H. Tanaka~\cite{Tanaka} for convex domains, P.\,L. Lions and A.\,S. Sznitman~\cite{LionsSznitman} for admissible domains, H. Saisho~\cite{SaishoSolEDS} for domains satisfying exterior sphere and interior cone conditions, P. Dupuis and H. Ishii~\cite {DupuisIshii1,DupuisIshii2} for non-smooth domains, and~\cite{MultipleConstraint} for domains defined by a finite set of constraints.
On the other hand, infinite-dimensional Brownian diffusions, without boundary condition, have been first studied by R. Lang~\cite{Lang1} and H.-O. Georgii~\cite{GeorgiiCanonicalGibbsMeasures}. 
Our present work combines the difficulties of infinite-dimensional Brownian systems and those induced by the reflection on the non-smooth boundary of the domain of non-overlapping configurations.

H. Tanemura~\cite{TanemuraEDS} first used such an infinite-dimensional reflected SDE to study a system of infinitely many identical Brownian hard balls. 
His results have been extended to identical hard balls that also have an additional smooth interaction with potentially infinite range (see~\cite{fradon_roelly_2000,fradon_roelly_tanemura_2000,FR2,Tanemura_2022}).
Recently, we extended these ideas in~\cite{FKRZ24} to two-type hard-core diffusions, in order to construct dynamics for the Asakura--Oosawa model. 
However, the results therein were limited to a finite number of large hard spheres in a bath of a infinitely many particles; as a result, the model in~\cite{FKRZ24}, though isotropic, does not have the spatial invariance which can be expected from a colloidal suspension at such a microscopic scale. 

The aim of the present paper is then to constructs a more natural, shift-invariant Asakura--Oosawa dynamics, where both the hard spheres and the particles are infinitely many. This natural generalisation is clearly not straightforward: since spheres, contrary to particles, are not penetrable, an infinite number of spheres induces infinite-range non-overlap conditions; this requires precise estimates (namely, to control so-called \emph{bad paths}) which lead to new, highly non-trivial, technical difficulties.

\medbreak

Colloids and the phenomenon of depletion interaction have recently been investigated using the tools and language of Statistical Mechanics. For example: in~\cite{JT19}, the authors use the emergence of the effective depletion interaction to show that cluster expansion converges in an improved activity regime; in~\cite{jahnel2024variational}, the Gibbs variational principle is shown to hold also when the size of hard spheres is unbounded; in~\cite{WJL22}, the authors analyse geometric criteria for the absence of multi-body depletion interactions. In this framework, our work shows how the Gibbs measures at the heart of the above studies come up naturally as the reversible equilibrium measures of the colloid-polymer dynamics.

\medbreak

The core of the paper is dedicated to the proof of existence of a unique strong solution to the two-type infinite-dimensional SDE that is introduced in Section~\ref{sec:model}: in Section~\ref{sec:construct}, we construct a sequence of approximating processes via penalisation, and in Section~\ref{sec:convergence}, we prove that such a sequence converges to the unique solution of the SDE. 
In Section~\ref{sec:depletion}, we analyse the emergence of the depletion interaction. More precisely, we consider the projection of the reversible measures onto the subsystem of hard spheres. We prove a correspondence between the two-type and the one-type Gibbs measures, and present an infinite-dimensional gradient diffusion associated to these measures. Finally, we show that in the low-activity regime there is absence of percolation, whereas the high-activity regime leads to the phenomenon of optimal packing.

\subsection{Allowed configurations of spheres and particles}
%%%%%%%%%%%%%%%%%%%%%%%%%%%%%%%%%%%%%%%%%%%%%%%%%%%%%%%%%%

We consider an infinite number of large spheres with radius $\rgrande>0$ and centres $\grande_1,\grande_2,\dots$ in $\R^d$, $d\geq 2$, along with an infinite number of smaller spherical particles with radius $0<\rpiccola<\rgrande$ and centres $\piccola_1,\piccola_2,\dots$ in $\R^d$. This two-type system is subject to the following \emph{non-overlap constraints}: the large spheres are not allowed to overlap; the small particles can overlap each other, but are not allowed to overlap the large ones.

The configuration space is the set $\M$ of locally finite Radon point measures on 
$\X = \R^d \times \{ \circ,\cdot \}$, i.e.\ those of the form
$\bx = \bgrande\bpiccola = \sum_{i\in I} \delta_{\grande_i} + \sum_{k\in K} \delta_{\piccola_k}$,
with
$\grande_i\in \R^d \times\{ \circ \}$, $\piccola_k\in\R^d \times\{ \cdot \}$ for $ I,K\subset\N^*=\N\setminus\{0\}=\{1,2,\dots\}$,
such that $\bx(\Lambda)<+\infty$ for any compact subset $\Lambda\subset\X$.
$\M$ is endowed with the topology of vague convergence. 

The sum of two point measures is denoted by the juxtaposition $\bx\by := \bx+\by$. As the point measures we consider are a.s.~simple, we identify them with their support.
For any $\Lambda\subset\R^d$, $\by_{\Lambda}:=\by\cap\Lambda$ is the restriction of the sphere and particle configuration $\by$ to $\Lambda$.
We denote by $\M_\Lambda$ the corresponding set of point measures on $\Lambda\times\{\circ,\cdot\}$ and by $\mu_\Lambda$ the restriction to $\M_\Lambda$ of a measure $\mu$ on $\M$. 
The above notations are, of course, analogous when considering only sphere or particle configurations, leading to $\MGrande_\Lambda$ and $\MPiccola_\Lambda$, respectively. We denote by $dx_\Lambda$ the Lebesgue measure on $\Lambda$.

According to the non-overlap constraints, the subset $\D\subset\M$ of so-called \emph{allowed configurations} is 
\begin{equation} \label{eq:D}
    \D = \Bigg\{ \bx=\bgrande\bpiccola \in \M :
    \begin{array}{rl}
        \forall i \neq j, \ &|\grande_i - \grande_j| \ge 2\rgrande, \\
        \forall i , k, \    &|\grande_i - \piccola_k| \ge \rgrande+\rpiccola 
    \end{array}
    \Bigg\}.
\end{equation} 
The second constraint describes a \emph{depletion} shell of thickness $\rpiccola$ around each sphere, that is forbidden for the centres of the small particles.
Given an allowed configuration $\bgrande$ of hard spheres, in order for the configuration $\bgrande\bpiccola$ to also be allowed, the particles $\bpiccola$ cannot be placed within the forbidden area 
$\BGrande(\bgrande) := \bigcup_{\grande\in\bgrande}B(\grande,\rdep)$,
where $\rdep:=\rgrande+\rpiccola$ is the depletion radius.

Here and in the sequel, $B(x,R)$ denotes the open ball with centre $x$ and radius $R$, $\vol{A}$ is the $d$-dimensional volume of $A\subset\R^d$, and $\volball_d:=\vol{B(0,1)}$ the volume of the $d$-dimensional unit sphere.

\subsection{Dynamics for the AO model} \label{sec:model}%%%%%%%%%%%%%%%%%%%%%%%%%%%%%%%%%%%%%%%%%%%%%%%%%%%%%%%
Having fixed notations, we can now introduce our dynamical model.
Fix a probability space $(\Omega,\mathcal{F},P)$; the system is described as follows:
\begin{itemize}
    \item Infinitely many \emph{hard spheres} with radius $\rgrande>0$, whose centres at time $t$ are denoted by $\big(\Grande_i(t)\big)_{i\in\N^*}$, move according to independent Brownian motions $\big(\WGrande_i\big)_{i\in\N^*}$. 
    \item The hard spheres evolve in a time-inhomogeneous {\em random medium} consisting of infinitely many \emph{small particles} with radius $\rpiccola\in(0,\rgrande)$, whose centres at time $t$ are denoted by $\big(\Piccola_k(t)\big)_{k\in\N^*}$, themselves moving according to independent Brownian motions $\big(\WPiccola_k\big)_{k\in\N^*}$.
    \item The only interactions are due to the non-overlap constraints.
    The \emph{local times} $(L_{ij})_{i,j \in \N^*}$ ensure that there is no pairwise overlap between the hard spheres: in case of a collision, they induce an instantaneous repulsion given by a normal reflection at the boundary of the set of allowed configurations.
    Similarly, the \emph{local times} $(\ell_{ik})_{i, k \in \N^*}$ model the non-overlap constraint between small particles and hard spheres.
    As a result, at each time, the two-type configuration should belong to the set $\D$ of allowed configurations.
\end{itemize}
This two-type dynamics can be described by the following doubly infinite SDE with reflection.
\begin{equation}\label{eq:SDE2infty} \tag{${\mathcal{S}_\infty}$}
\begin{cases}
    \text{for any } i\in\N^*, \, k\in\N^*, \,  t\in\R_+=[0,+\infty) ,                    \\ \disp  
    d\Grande_i(t) 
    = d\WGrande_i(t) 
      +\sum_{j=1}^{+\infty} \big(\Grande_i-\Grande_{j}\big)(t) dL_{ij}(t) 
       +\sum_{k=1}^{+\infty} \big(\Grande_i-\Piccola_{k}\big)(t) d\ell_{ik}(t) \\ \disp 
    d\Piccola_k(t) 
    = \sigmaP\, d\WPiccola_k(t) 
      +\sigmaP^2 \sum_{j=1}^{+\infty} \big(\Piccola_k-\Grande_j\big)(t)  d\ell_{ik}(t) \\ \disp
    \text{for any } j\in\N^*,\ L_{ij}(0)=\ell_{ik}(0)=0, \ L_{ij}=L_{ji},     \\ \disp
    \int_0^t \un_{|\Grande_i(s)-\Grande_j(s)|\neq2 \,\rgrande} \ dL_{ij}(s)=0,\ L_{ii}\equiv 0,\\ 
    \disp \int_0^t \un_{|\Grande_i (s)-\Piccola_k(s)|\neq\rdep} \, d \ell_{ik}(s)=0.
\end{cases}
\end{equation}
The \emph{diffusion coefficient} $\sigmaP>0$ is fixed, related to the temperature and to the mass of the small particles. We do not introduce a corresponding coefficient $\mathring{\sigma}$ for the hard spheres, i.e.\ we simply set it equal to $1$; this is not a restriction, as it suffices to choose, as the time unit of the model, the time at which the standard deviation of each coordinate of the Brownian motions $\WGrande_i$ is equal to the length unit of the model.

\subsection{Main results} %%%%%%%%%%%%%%%%%%%%%%%%%%%%%%%%%%%%%%%%%%%%%%%%%
%%%%%%%%%%%%%%%%%%%%%%%%%%%%%%%%%%%%%%%%%%%%%%%%%%%%%%%%%%%%%%%%%%%%%%%%%%

The following set of Gibbs measures will be the reversible measures for the above system.
\begin{definition}
    For a probability measure $\mu$ on $\M$, we write $\mu\in\Gcal_{\zgrande,\zpiccola}$
    and say that it is a \emph{Gibbs measure} on $\M$, with activities $\zgrande,\zpiccola>0$, associated to the above hard-core interaction, if, for any bounded subset $\Lambda\subset\R^d$ and any positive measurable function $F$ on $\D$,
    \begin{equation} \label{eq:mu}
    \begin{split}
        &\int_\D F(\bx) \, \mu(d\bx) 
        := \int_\D \frac{1}{Z_\Lambda(\by)} \int_{\MGrande} \int_{\MPiccola} 
       F(\by_{\Lambda^c} \bgrande\bpiccola) \, 
       \un_{\D}(\by_{\Lambda^c} \bgrande\bpiccola) \, 
       {\pi^\zpiccola_\Lambda}(d\bpiccola)  \, {\pi^\zgrande_\Lambda}(d\bgrande) \, \mu(d\by),
    \end{split}
    \end{equation}
    where $\pi^\zgrande_\Lambda(d \bgrande)$ denotes the Poisson point process on $\MGrande_\Lambda$ with intensity measure $\zgrande dx_\Lambda$, $\pi^\zpiccola_\Lambda(d \bpiccola)$ denotes the Poisson point process on $\MPiccola_\Lambda$ with intensity measure $\zpiccola dx_\Lambda$, and the normalisation factor $Z_\Lambda(\by)= Z_\Lambda(\by_{\Lambda^c}) = \iint \un_{\D}(\by_{\Lambda^c} \bgrande\bpiccola) \, {\pi^\zpiccola_\Lambda}(d\bpiccola) \, {\pi^\zgrande_\Lambda}(d\bgrande)$ is the partition function.
\end{definition}
\begin{remark}
   Note that $\mu\in\Gcal_{\zgrande,\zpiccola}$ is concentrated on the set $\D$ of admissible configuration. Also note that $Z_\Lambda(\by)\leq 1$, since $\un_{\D}\le1$,
and $Z_\Lambda(\by)>0$, since $\by\in\D$ and $\bgrande\bpiccola$ is a finite configuration.

While the existence of at least one such measure, for any choice of the activities $\zgrande$ and $\zpiccola$, is a classical result (see, e.g.~\cite{Ruelle_book}), it is conjectured~\cite[Chapter 3]{LT11} that a non-uniqueness phase transition, that is $\text{Card}(\Gcal_{\zpiccola,\zgrande})>1$, occurs for large enough hard-sphere activities. 
\end{remark}

\medbreak

The first main result of this paper is the following theorem, whose proof is split in Sections~\ref{sec:construct} and~\ref{sec:convergence}. See Section~\ref{sec:strategy} for the proof strategy.
\begin{theorem}\label{theorem:SDE2infty}
    For any values of the activities $\zgrande,\zpiccola>0$, for any Gibbs measure $\mu\in\Gcal_{\zgrande,\zpiccola}$, the two-type infinite-dimensional SDE~\eqref{eq:SDE2infty} admits a unique 
    $\D$-valued strong solution, for $\mu$-almost every deterministic initial condition. Moreover, this solution is time-reversible
    if the initial condition is random and distributed according to $\mu$. 
\end{theorem}

In Section~\ref{sec:depletion}, we investigate the emergence of an effective depletion interaction. The main results can be summarised in the following two theorems.
\begin{theorem}
    For any values of the activities $\zgrande,\zpiccola>0$, the sphere marginal probability measure $\mugrande$ of any $\mu\in\Gcal_{\zgrande,\zpiccola}$ is an element of $\Gcal_{\zgrande}(\zpiccola\nrg)$, i.e.\ it solves, for any bounded $\Lambda\subset\R^d$ and any positive measurable $F$ on $\MGrande$,
    \begin{equation}\label{eq:DLR_dep}
    \begin{split}
        \int_\D F(\bgrande) \, \mugrande(d\bgrande) = \int_\D \frac{1}{Z_\Lambda(\bgrandey)} \int_{\MGrande} 
        F(\bgrande \bgrandey_{\Lambda^c}) \, 
        e^{-\zpiccola\nrg_\Lambda(\bgrande\bgrandey_{\Lambda^c})}\un_{\D}(\bgrandey_{\Lambda^c} \bgrande) \, 
        \pi^\zgrande_{\Lambda}(d\bgrande) \, \mugrande(d\bgrandey),
    \end{split}
    \end{equation} 
    where $\nrg_\Lambda (\bgrande) := \vol{\BGrande(\bgrande_\Lambda)\setminus \BGrande(\bgrande_{\Lambda^c})}$ has interaction range $2\rdep$.
    Moreover, the following behaviour is observed:
    
    \emph{(Low activity)} 
    There exists a critical activity $\zgrande_{\rm c}>0$ such that, for any $\zgrande<\zgrande_{\rm c}$, independent of any $\zpiccola>0$, $\mugrande$-almost surely, there is no percolation of interacting (i.e.\ at distance $2\rdep$) hard spheres;
    
    \emph{(High activity)} For any $\zpiccola>0$, any measure $\mugrande\in\Gcal_{\zgrande}(\zpiccola\nrg)$ attains the closest-packing density as $\zgrande\to\infty$.
\end{theorem}
It is worth noting here that the above absence of percolation result in the low activity regime (Proposition~\ref{prop:No percolation} below) relies on the same estimates needed to control the pathologies in the dynamics (as described in the next section) required for proving Theorem~\ref{theorem:SDE2infty}.

\begin{theorem}\label{th:DepletionDynamics}
Let $\nrg (\bgrande) := \displaystyle \vol{\BGrande(\bgrande)}$. 
    For $d\geq 3$ and $\rpiccola/\rgrande\leq \frac{2}{3}\sqrt{3}-1$, for any values $\zgrande,\zpiccola>0$, the following gradient dynamics admits a unique strong solution, and $\mugrande\in\Gcal_{\zgrande}(\zpiccola\nrg)$ is reversible for this dynamics:
    \begin{equation}\tag{${\mathcal{S}}^\textrm{dep}$}\label{eq:Sdep}
    \begin{cases}
    \begin{array}{l}
        \textrm{for any } i\in\N^*, \, t \in \R^+ ,                             \\ 
        d \Grande_i (t) =
        d \WGrande_i(t)
        - \disp \frac{\zpiccola}{2} \,
        \nabla_i\nrg(\Grande_1,\Grande_2,\dots)\, d t + \sum_{j=1}^{+\infty}\big(\Grande_i-\Grande_j\big)(t) d L_{ij}(t) \, , \\
        \disp
        \text{for any } j\in\N^*,\ L_{ij}(0) = 0,\ L_{ij} = L_{ji},\\
        \disp \int_0^t \un_{|\Grande_i (s)-\Grande_j(s)|\neq 2 \,\rgrande} \, d  L_{ij}(s) = 0, \ L_{ii} \equiv 0.
    \end{array}
    \end{cases}	
    \end{equation}
\end{theorem}
\begin{remark}
    The effect of the depletion interaction $\nrg$ is stronger the larger $\zpiccola$ is, but both the measure and the associated dynamics exist for any value of the activity, cf. Section~\ref{sec:depletionSDE}.
    Moreover, because of the hard-sphere exclusion, the average number of spheres per unit volume is bounded from above, independently of the value of the activity $\zgrande$, which can take arbitrary values.
\end{remark}
Note that the solutions of the above theorems are unique in the sense of~\cite[Lemma 5.4]{TanemuraEDS}, that is as elements of a set of regular paths, see Remark~\ref{rmk:uniqueness}.

\begin{remark}
    For the two SDEs above,~\eqref{eq:SDE2infty} and that of Theorem~\ref{th:DepletionDynamics}, since the dynamics are Markovian, in order to construct a strong solution on $\R_+$, it is enough to construct one on the time interval $[0,1]$. Indeed, the solution on $[0,2]$ can be obtained by glueing together the solution on $[0,1]$ with a new dynamics on $[0,1]$ started from the final configuration of the first one. This can be iterated to extend the solution to the whole of $\R_+$. In what follows, then, we will construct the solutions only on $[0,1]$, as the existence of solutions on $\R_+$ follows.
\end{remark}

%>>>>>>>>>>>>>>>>>>>>>>>>>>>>>>>>>>>>>>>>>>>>>>>>>>>>>>>>>>>>>>>>>>>>>>>>>>>>>>>
\subsection{Strategy of proof}\label{sec:strategy}
%>>>>>>>>>>>>>>>>>>>>>>>>>>>>>>>>>>>>>>>>>>>>>>>>>>>>>>>>>>>>>>>>>>>>>>>>>>>>>>>

The following two sections are dedicated to the proof of Theorem~\ref{theorem:SDE2infty}. As the proof is quite technical, requiring very precise control of rare pathological events, we present here an extended proof sketch. 

The underlying idea is relatively simple: even though~\eqref{eq:SDE2infty} is infinite-dimensional because of the infinite number of balls (spheres and particles), locally, it should be essentially finite-dimensional, in the sense that any Brownian ball interacts (with high probability) only with a finite number of other spheres or particles. More precisely, since Brownian motion is relatively slow (moving at speed $\sqrt{t}$), we should be able to assume that, over a small time interval, the trajectory of a ball depends exclusively on its own Brownian motion and the collisions due to neighbouring balls, safely ignoring far-away ones, as the probability that they cover their distance is negligible. With this reasoning, we can then divide the time interval $[0,1]$ in finitely many short time intervals in which the balls interact only in finite groups of controlled size, thus reverting to the framework of finite-dimensional SDEs, in which the existence of solutions is well-known. 

For this simple strategy to work, not only should the Brownian balls not move too fast, but also far-away balls should not influence each other too quickly in any other way; this is unfortunately not true in general. 
Indeed, a domino-like effect can cause distant balls to interact almost instantaneously if there is a collision along a long chain of spheres that are very close to each other. In order to rigorously follow the above reasoning, we then need to show that these two pathologies -- too-fast Brownian balls and too-long chains of spheres --  almost never occur; this is the main technical difficulty.

In more detail, the main steps of the proof are as follows:
\medbreak

\textbf{Step 1: Construction of penalised finite-dimensional dynamics.}
In Section~\ref{sec:construct}, we consider, for any initial configuration $\by=\grandey\piccolay$ and for any ball of radius $R>0$ around the origin, penalising functions $\psiGyR$ and $\psiPyR$ that will keep around $B(0,R)$ the spheres and particles initially placed inside $B(0,R)$ while also penalising them for getting too close to balls outside $B(0,R)$, cf.~\eqref{eq:psiGR}. Letting move and collide only those balls initially placed in $B(0,R)$, while keeping frozen those outside, we obtain the finite-dimensional SDE~\eqref{eq:SDEnSpheres}, in which the influence of the infinitely many frozen balls is felt only via the drifts $\nabla\psiGyR$ and $\nabla\psiPyR$. For any number $n$ of moving spheres and $m$ of moving particles,~\eqref{eq:SDEnSpheres} admits a unique strong solution $[0,1]\ni t\mapsto\bX^{\by,R,\ngrande,\npiccola}(\bx,t)$ for all initial configurations $\bx$ (arbitrary or coinciding with $\by_{B(0,R)}$).
\medbreak

\textbf{Step 2: Reversible approximations and consistent initial conditions.}
In order for the finite system to converge towards a solution of the infinite-dimensional SDE~\eqref{eq:SDE2infty}, two properties are needed: (1) Being at equilibrium, i.e. time-reversible, to obtain the estimates in total variation needed for the convergence; (2) Having consistent initial configurations, coming from the restrictions of the same infinite $\by$ to $B(0,R)$ as we let $R$ tend to infinity.
    
Since the reversible law of each finite system depends on the specific $R$, it is clear that the two properties, as stated, cannot hold at the same time. In Section~\ref{sec:mixture}, we construct the reversible solution of~\eqref{eq:SDEnSpheres} whose law $Q^\by_R$ has as marginal at each instant the equilibrium law $\mu^\by_R$ of the finite-dimensional system penalised on $B(0,R)$. We will compare it with solutions $\bX^{\by,R}$ of~\eqref{eq:SDEnSpheres} with starting configuration $\by_{B(0,R)}$ that, integrated over $\mu(d\by)$, asymptotically yield a reversible solution of~\eqref{eq:SDE2infty}, and finally show, in Step 4 below, that the two solutions coincide as $R\to\infty$.
\medbreak

\textbf{Step 3: Path-by-path convergence on a subset.}
In Section~\ref{sec:ball-separation}, we fix an initial configuration $\by$ and $\omega\in\Omega$, hence fixing a realisation of all the Brownian motions. We consider the solutions $\bX^{\by,R}$ and $\bX^{\by,R+1}$ of the finite-dimensional dynamics in $B(0,R)$ and $B(0,R+1)$, respectively. We show that, if $\omega$ does not induce very long chains or very fast balls, when $R$ is large enough, the trajectories of balls close to the origin do not feel the penalisation and coincide under the two penalised processes. This yields, for the trajectory of each sphere or particle, a stationary (thus converging) sequence. These trajectories are then solutions to~\eqref{eq:SDEnSpheres} without penalisation, hence they solve~\eqref{eq:SDE2infty}. We call $\Omega^\by$ the subset of $\Omega$ of such good realisations, and note that it depends on the choice of several parameters, cf.~\eqref{eq:defOmegay}, which play a central role in the next step.
\medbreak

\textbf{Step 4: Pathologies are negligible.}
In Section~\ref{sec:almost-sure}, we show that the subset $\Omega^\by\subset\Omega$ on which we have the convergence is of full mass for $\mu$-almost every initial configuration $\by$. Since $\omega\notin\Omega^\by$ means that the previously mentioned pathologies are present for an infinite number of $R$'s, thanks to the Borel--Cantelli lemma, this reduces to showing convergence of a series. The only difficulty amounts then to a careful choice of the afore-mentioned parameters.

Under the reversible laws $Q^\by_R$, it turns out that chains longer than $\kappa(R):=\floor{R^{1/3}}$ of spheres, in or around $B(0,R)$, closer than $\eps$ from each other at times multiples of $\delta(R):=1/\floor{R^{1/3}}$ are exponentially unlikely as $R$ increases. Similarly for the probability that a ball going through $B(0,R)$ moves by more than $\eps/4$ in the time interval of length $\delta(R)$. This informs the choice of the parameters in $\Omega^\by$. The penalisation is then strong enough for the total variation distance between the laws $Q^\by_R$ and the law of $\bX^{\by,R}$ with $\by\sim\mu$ to converge at summable speed, completing the proof.

\begin{remark}
    Note that the above proof does not depend on the choice of intensity for the reference Poisson point process, cf.~\eqref{eq:mu}, and hence our result holds for all values $\zgrande,\zpiccola>0$.
\end{remark}

%>>>>>>>>>>>>>>>>>>>>>>>>>>>>>>>>>>>>>>>>>>>>>>>>>>>>>>>>>>>>>>>>>>>>>>>>>>>>>>>
\section{Construction of the approximating process}\label{sec:construct}
%>>>>>>>>>>>>>>>>>>>>>>>>>>>>>>>>>>>>>>>>>>>>>>>>>>>>>>>>>>>>>>>>>>>>>>>>>>>>>>>

Let $\mu\in\Gcal_{\zgrande,\zpiccola}$, $\zgrande,\zpiccola>0$. In this section, for $\mu$-almost every initial condition $\by\in\D$, we construct, via a penalisation procedure, a sequence of approximating processes that start from an approximation $\mu^\by_R$ of the candidate reversible measure for~\eqref{eq:SDE2infty}, and that mostly stay in $B(0,R)$.
The solution of~\eqref{eq:SDE2infty} will then be shown, in Section~\ref{sec:convergence}, to be the limit as $R\to\infty$ of such a sequence.

\subsection{Penalisation functions}\label{sebsect:penalisationFunctions} %%%%%%%%%%%%%%%%%%%%%%

For every fixed radius $R>0$ and outside configuration $\by=\bgrandey\bpiccolay\in\D$, we define $\BallGyR$ (resp. $\BallPyR$) as the part of $B(0,R)$ in which the centre of a sphere (resp. a particle) can be put without conflicting with the $\by$ spheres and particles whose centres are outside of $B(0,R)$, that is:
\begin{equation}\label{eq:BallGyR_BallPyR} 
\begin{split}
   &\BallGyR := B(0,R) \setminus B(\bgrandey_{B(0,R)^c}, 2\rgrande)
                      \setminus B( \bpiccolay_{B(0,R)^c} , \rdep), \\
   &\BallPyR := B(0,R) \setminus \BGrande(\bgrandey_{B(0,R)^c}).
\end{split}
\end{equation}
The sphere-penalisation function $\psiGyR : \R^d \to \R$ 
and the particle-penalisation function $\psiPyR : \R^d \to \R$ 
are chosen among non-negative functions of class $\mathcal{C}^2$ with bounded derivatives, 
such that $\nabla\psiGyR$ vanishes on $\BallGyR$, $\nabla\psiPyR$ vanishes on $\BallPyR$,
and they are small enough outside of $\BallGyR$ and $\BallPyR$ respectively, in the sense that:
\begin{equation} \label{eq:psiGR}
\begin{split}
    \sum_{R=1}^\infty \int_{(\BallGyR)^c} e^{-\psiGyR(x)} \, dx <+\infty
  \quad\text{ and }\quad \sum_{R=1}^\infty \int_{(\BallPyR)^c} e^{-\psiPyR(x)} \, dx <+\infty.
\end{split}
\end{equation}
That is, for the finite measures on $\R^d$ defined as
$\lambdaGyR(d\grande) := e^{-\psiGyR(\grande)} \,d\grande$, 
and $\lambdaPyR(d\piccola) := e^{-\psiPyR(\piccola)} \,d\piccola$,
where $dx$ denotes the Lebesgue measure, we assume that 
$\sum_{R=1}^\infty ~ \lambdaGyR((\BallGyR)^c) + \lambdaPyR((\BallPyR)^c) ~ <+\infty.$
An example of such functions $\psiGyR$ and $\psiPyR$\,, having linear growth with respect to the Euclidean norm at infinity, is given in Section~\ref{app:psigrandepsipiccola}.

These penalisation functions will be used as drifts in a finite-dimensional penalised dynamics whose solution will be shown to approximate the solution of~\eqref{eq:SDE2infty}. 

\subsection{Penalised SDE} %%%%%%%%%%%%%%%%%%%%%%%%%%%%%%%%%%%%%%%%%%%%%%

We consider the penalised SDE around $B(0,R)$ for a finite number $\ngrande$ of spheres and a finite number $\npiccola$ of particles given by
\begin{equation}\label{eq:SDEnSpheres} \tag{${\mathcal{S}^\by_{R}}$}
\begin{cases}
\text{for any } i\in\{1,\ldots,\ngrande\}, \, k\in\{1,\ldots,\npiccola\}, \, t\in[0,1], \\ \disp  
 \Grande_i(t) = \Grande_i(0) +\WGrande_i(t) -\frac12 \int_0^t \nabla\psiGyR\big(\Grande_i(s)\big) \,ds \\ 
 \disp \phantom{ \Grande_i(t) = \Grande_i(0) +\WGrande_i(t) }
                +\sum_{j=1}^\ngrande \int_0^t \big(\Grande_i-\Grande_{j}\big)(s) dL_{ij}(s) 
                +\sum_{k=1}^{\npiccola} \int_0^t \big(\Grande_i-\Piccola_{k}\big)(s) d\ell_{ik}(s) \\ 
                \disp 
 \Piccola_k(t) = \Piccola_k(0) +\sigmaP\,\WPiccola_k(t) 
                 -\frac{\sigmaP^2}{2} \int_0^t \nabla\psiPyR\big(\Piccola_k(s)\big) \,ds
                +\sigmaP^2 \sum_{i=1}^\ngrande \int_0^t \big(\Piccola_k-\Grande_{i}\big)(s)  d\ell_{ik}(s) \\ \disp
    \text{for any } j\in\{1,\ldots,\ngrande\},\ L_{ij}(0)=\ell_{ik}(0)=0, \ L_{ij}\equiv L_{ji}, \\ 
    \disp 
    \int_0^t \un_{|\Grande_i(s)-\Grande_j(s)|\neq2 \,\rgrande} \ dL_{ij}(s)=0, \
    L_{ii}\equiv 0,\\ 
  \disp\int_0^t \un_{|\Grande_i (s)-\Piccola_k(s)|\neq\rdep} \, d \ell_{ik}(s)=0 .
\end{cases}
\end{equation}
The existence of a unique solution to the above system follows from~\cite[Proposition 2.2]{FKRZ24}:
\begin{proposition} \label{prop:ExistenceSolRevSnmR}
For any fixed outside condition $\by\in\D$, any fixed radius $R>0$, and any number $\ngrande\in\N$ of large spheres and $\npiccola\in\N$ of small particles,
the penalised finite-dimensional SDE with reflection~\eqref{eq:SDEnSpheres} admits a unique $\D$-valued strong solution
$\bX^{\by,R,\ngrande,\npiccola}(\bx) = \big(\bX^{\by,R,\ngrande,\npiccola}(\bx,t),\, t\in [0,1]\big)$,
\begin{equation*}
\begin{split}
     &\bX^{\by,R,\ngrande,\npiccola}(\bx,t) :=
    \big( \Grande_i^{\by,R,\ngrande,\npiccola}(\bx,t), \Piccola_k^{\by,R,\ngrande,\npiccola}(\bx,t),  L_{ij}^{\by,R,\ngrande,\npiccola}(\bx,t), \ell_{ik}^{\by,R,\ngrande,\npiccola}(\bx,t) \big)_{1\le i,j \le \ngrande,\, 1\le k \le \npiccola},
\end{split}
\end{equation*}
for $\nu^\by_{R,\ngrande,\npiccola}$-almost every deterministic initial condition\\ $\bx=\bgrande\bpiccola= (\grande_1,\ldots,\grande_n,\piccola_1,\ldots,\piccola_\npiccola) \in \D$, 
where the measure $\nu^\by_{R,\ngrande,\npiccola}$, concentrated on the admissible configurations with $\ngrande$ spheres and $\npiccola$ particles, is defined, for any positive measurable function $F$ on $\D$, by 
\begin{equation} \label{eq:nu_yRn}
\begin{split}
    &\int_\D F(\bx) \, d\nu^\by_{R,\ngrande,\npiccola}(\bx):= \int_{\R^{\ngrande d}} \int_{\R^{\npiccola d}} 
       F(\bgrande\bpiccola) \, \un_{\D}(\bgrande\bpiccola) \, 
       \otimes_{k=1}^\npiccola \lambdaPyR(d\piccola_k) \otimes_{i=1}^\ngrande \lambdaGyR(d\grande_i).
\end{split}
\end{equation}
Moreover, the finite measure $\nu^\by_{R,\ngrande,\npiccola}$ 
% and the associated probability measure $\frac{1}{\Zmixing^\by_{R,\ngrande,\npiccola}} \nu^\by_{R,\ngrande,\npiccola}$ 
is reversible for the dynamics~\eqref{eq:SDEnSpheres}.
\end{proposition}

Note that $\nu^\by_{R,\ngrande,\npiccola}$ only depends on $\by$ via $\psiGyR$ and $\psiPyR$, which only depend on $\by_{B(0,R)^c}$. 
% It does not depend on $\bgrandey_{B(0,R)}$ nor $\bpiccolay_{B(0,R)}$. 
The initial configurations $\bx=\bgrande\bpiccola$ given by $\nu^\by_{R,\ngrande,\npiccola}$ belong to $\D$, but the superposition $\by_{B(0,R)^c}\bgrande\bpiccola$ is in general not in $\D$.

\begin{remark}
%[On the cases where there is no particle or no sphere]
For $\ngrande=0$, there is no sphere process and no local time process $L_{ij}$ and $\ell_{ik}$;
the process $\bX^{\by,R,0,\npiccola}(\bx,\cdot) = \big( \Piccola_k^{\by,R,0,\npiccola}(\bx,\cdot)\big)_{1\le k \le\npiccola} $ only involves $\npiccola$ independent particles. 
Then~\eqref{eq:nu_yRn} holds with the convention that $\bgrande=\emptyset$, 
$\int_{\R^{0}}  \const \otimes_{i=1}^0 \lambdaGyR(d\grande_i) = \const$, and $\int F(\bx) \, d\nu^\by_{R,0,\npiccola}(\bx) 
    := \int_{\R^{\npiccola d}} F(\bpiccola) \, 
       \otimes_{k=1}^\npiccola \lambdaPyR(d\piccola_k)$.
       
Similarly, for $\npiccola=0$, there is no particle process and no local time process $\ell_{ik}$; the process $\bX^{\by,R,\ngrande,0}(\bx,t) = \big( \Grande^{\by,R,\ngrande,0}(\bx,\cdot), L_{ij}^{\by,R,\ngrande,0}(\bx,\cdot) \big)_{ 1\le i,j \le \ngrande} $ only involves $\ngrande$ hard spheres. 
Then~\eqref{eq:nu_yRn} holds with the convention that $\bpiccola=\emptyset$ and 
$\int F(\bx) \, d\nu^\by_{R,\ngrande,0}(\bx) 
    := \int_{\R^{\ngrande d}} 
       F(\bgrande) \, \un_{\D}(\bgrande) \, \otimes_{i=1}^\ngrande \lambdaGyR(d\grande_i)$.
       
For $\ngrande=0$ and $\npiccola=0$, the process is empty. For the sake of completeness, in the sequel we let $\int F(\emptyset) \, d\nu^\by_{R,0,0}(\bx) := F(\emptyset)$.
\end{remark}

\begin{definition} 
Let $Q^\by_{R,\ngrande,\npiccola}$ denote the time-reversible pushforward measure of the solution of the penalised SDE~\eqref{eq:SDEnSpheres} with initial measure $\nu^\by_{R,\ngrande,\npiccola}$:
\begin{equation} \label{eq:Q_yRn}
    Q^\by_{R,\ngrande,\npiccola}\left( ~\cdot~ \right) 
    := \int P\left( \big( \bX^{\by,R,\ngrande,\npiccola}(\bx,t) \big)_{t\in[0,1]} \in~\cdot~ \right) \, d\nu^\by_{R,\ngrande,\npiccola}(\bx). 
\end{equation}
\end{definition} 

\subsection{Mixture of penalised processes} \label{sec:mixture}
%%%%%%%%%%%%%%%%%%%%%%%%%%%%%%%%%%%%%%%%%%%%%%

We construct an approximation of the expected reversible measure $\mu$ as a Poissonian mixture of the reversible measures of the solutions of~\eqref{eq:SDEnSpheres} for different numbers of spheres and particles:
\begin{equation}\label{eq:muRy}
\mu^\by_R := \frac{1}{\Zmixing^\by_R} 
             \sum_{\ngrande=0}^{+\infty} \frac{\zgrande^\ngrande}{\ngrande!} \,                
             \sum_{\npiccola=0}^{+\infty} \frac{\zpiccola^\npiccola}{\npiccola!} \, \nu^\by_{R,\ngrande,\npiccola}.
\end{equation} 
For any $\by$, the following mixture $Q^\by_R$ of penalised processes distributions starting from $\mu^\by_R$ is reversible:
\begin{equation} \label{eq:QRy}
Q^\by_R := \frac{1}{\Zmixing^\by_R} 
           \sum_{\ngrande=0}^{+\infty} \frac{\zgrande^\ngrande}{\ngrande!} \,                
           \sum_{\npiccola=0}^{+\infty} \frac{\zpiccola^\npiccola}{\npiccola!} \, Q^\by_{R,\ngrande,\npiccola}.
\end{equation} 
Finally, we define the approximate process $X^{\by,R}$ as the solution of~\eqref{eq:SDEnSpheres} with $\ngrande=\sharp(\bgrandey_{B(0,R)})$ and $\npiccola=\sharp(\bpiccolay_{B(0,R)})$, where $\sharp$ denotes the number of points, and with initial configuration equal to $\bgrandey_{B(0,R)}\bpiccolay_{B(0,R)}$, that is,
\begin{equation}\label{eq:XyR}
    X^{\by,R}(\cdot) := \bX^{\by,R,\sharp(\bgrandey_{B(0,R)}),\sharp(\bpiccolay_{B(0,R)})}(\by_{B(0,R)},~\cdot~).
\end{equation}
To obtain Theorem \ref{theorem:SDE2infty}, we will prove that, as $R$ tends to infinity, the $Q^\by_R$-distributed reversible process converges, in a locally eventually constant way, to a solution of~\eqref{eq:SDE2infty}, and the distribution of $X^{\by,R}$ is arbitrary close to $Q^\by_R$ when the outside configuration $\by$ is chosen according to $\mu$.

\section{Convergence of the penalised processes}\label{sec:convergence}
%%%%%%%%%%%%%%%%%%%%%%%%%%%%%%%%%%%%%%%%%%%%%%

In this section, we show that the penalised process $X^{\by,R}$ converges to the solution of~\eqref{eq:SDE2infty}: we first prove, in Section~\ref{sec:ball-separation}, that the convergence holds in a certain subset of $\Omega$, and then, in Section~\ref{sec:almost-sure}, show that this subset has full mass under the reversible measures. Uniqueness is shown in Remark~\ref{rmk:uniqueness}.
\medbreak

In order to prove the convergence, we have to check that balls (spheres or particles) only interact locally, making sure that the probability of \emph{bad paths}, i.e.\ those of balls that interact with a great number of other ones, vanishes.
The proof relies on the fact that \emph{long chains} of interacting spheres rarely form, and that balls coming \emph{very fast} from far away also are improbable.
Let us first define these two events. See Figure~\ref{fig:chain} for a visualisation of a long chain. 
\begin{figure}
    \centering
    \includegraphics[width=0.8\linewidth]{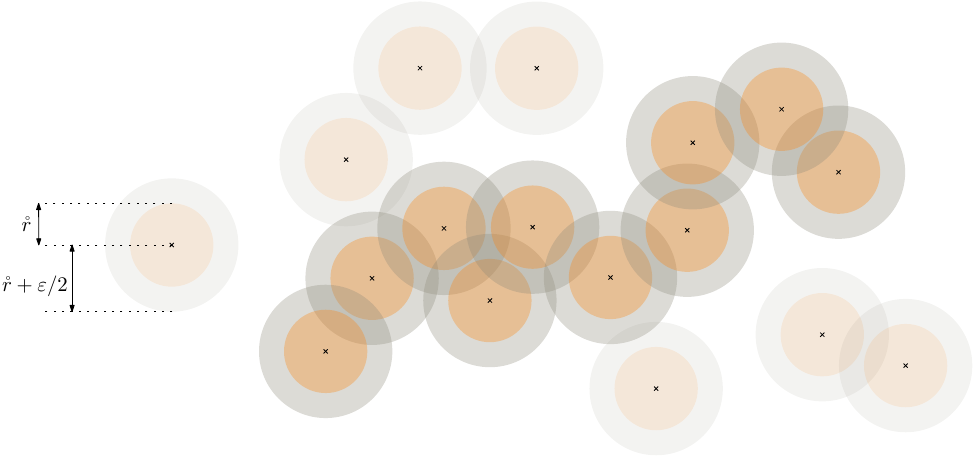}
    \caption{Highlighted, an example of an $\eps$-chain of length $\kappa = 9$.}
    \label{fig:chain}
\end{figure}
\begin{definition}[Long chains]\label{def:defChain}
For every positive radius $\alpha$, every positive leeway $\eps$ and every chain length $\kappa\in\N^*$, define the set of configurations which have an $\epsilon$-chain of spheres of length $\kappa$, with one end in $B(0,\alpha)$, as
\begin{align*} 
    \BadPath_{\text{Chain}}&(\alpha,\kappa,\eps)
    := \Bigg\{ \bx \in \D 
                \text{ s.t. for some } \{ \grande_{j_1},\cdots, \grande_{j_{\kappa+1}} \} \subset \bgrande, \\
                & |\grande_{j_1}|\le\alpha \text{ and } 
                |\grande_{j_1}-\grande_{j_2}|<2\rgrande+\eps,~ 
                \ldots ,~
                |\grande_{j_{\kappa}}-\grande_{j_{\kappa+1}}|<2\rgrande+\eps 
       \Bigg\} .
\end{align*}
For every positive time step $\delta$, the set of configuration paths whose $\delta$-discretisation presents an $\eps$-chain is
\begin{align*}
    &\BadPath'_{\text{Chain}}(\delta,\alpha,\kappa,\eps):= \Bigg\{ \bX\colon [0,1]\to \D 
                 \text{ s.t. } \exists k\in\N : \bX(k\delta)\in\BadPath_{\text{Chain}}(\alpha,\kappa,\eps)
       \Bigg\} .
\end{align*} 
\end{definition} 
\begin{definition}[Fast balls]\label{def:defFast}
For every positive radius $\alpha$, every positive range $\eps$, and every positive time duration $\delta$, define the set of paths which have oscillation $\eps$ during a time interval smaller than $\delta$, and are at some time point in $B(0,\alpha)$, as
\begin{equation*}% \label{eq:defFast}
\begin{split}
    &\BadPath_{\text{Fast}}(\alpha,\delta,\eps)\\
    &:= \bigg\{ f\in \mathcal{C}^0\big([0,1];\R^d\big) \text{ s.t. } 
               \min_{0\le s \le1}|f(s)| \leq \alpha
    \text{  and } 
               \sup_{\substack{|t-s|<\delta \\ 0\le s,t \le1}} |f(t)-f(s)| > \eps \bigg\} .  
\end{split}
\end{equation*} 
The set of configuration paths involving a sphere or particle that moves too fast is defined as
\begin{equation*}%\label{eq:defOmegaFast}
\begin{split}
    &\BadPath'_{\text{Fast}}(\alpha,\delta,\eps)\\
    &:= \bigg\{ \bX\colon [0,1]\to \D 
            \text{ s.t. } \exists i: 
            \Grande_i(\cdot)\in\BadPath_{\text{Fast}}(\alpha,\delta,\eps)
            \text{ or }
            \exists k:
            \Piccola_k(\cdot)\in\BadPath_{\text{Fast}}(\alpha,\delta,\eps) \bigg\} .
\end{split}
\end{equation*}
\end{definition}
We will show that the approximate processes $X^{\by,R}$ coincide for $R$ large enough, as far as their spheres and particles in some bounded area are concerned. 
This happens as soon as, for any $\eps$ smaller than some positive threshold $\eps(\by)$ 
and for all ranges $R$ large enough, both processes $X^{\by,R}$ and $X^{\by,R+1}$ do not belong to the $\BadPath'_{\text{Chain}}$ and $\BadPath'_{\text{Fast}}$ events given by $\eps$ and some suitable parameters $\delta$, $\alpha$, and $\kappa$ depending on $R$.
Formally, the set on which the convergence holds, and which will be shown to be of full measure, is:
\begin{equation} \label{eq:defOmegay}
\begin{split}
    \Omega^\by := \bigg\{ 
    & \omega\in\Omega : 
    \exists\eps(\by)>0 \ \forall\eps<\eps(\by) \ 
    \exists R(\eps)\in\N \text{ s.t. } \forall R\ge R(\eps),   \\
    & X^{\by,R}(\omega) \notin \BadPath'_{\text{Chain}}(\delta(R),\alpha(R),\kappa(R),\eps), \ X^{\by,R} \notin \BadPath'_{\text{Fast}}(\alpha(R),\delta(R),\eps/4), \\
    & X^{\by,R+1}(\omega) \notin \BadPath'_{\text{Chain}}(\delta(R),\alpha(R),\kappa(R),\eps),\ X^{\by,R+1} \notin \BadPath'_{\text{Fast}}(\alpha(R),\delta(R),\eps/4)
                \bigg\},
\end{split}
\end{equation} 
where 
$\delta(R):=1/\floor{R^{1/3}}$, $\alpha(R):=R-2\rgrande$ and $\kappa(R):=\floor{R^{1/3}}$.

The construction of a solution to~\eqref{eq:SDE2infty} is then split into the two following propositions, proved in Sections~\ref{sec:ball-separation} and~\ref{sec:almost-sure} respectively, and a short Section~\ref{sec:reversibility} to check the reversibility of the limit process obtained in these propositions.

\begin{proposition}\label{prop:Convergence_y}
    For any fixed $\by$ in the interior of $\D$, for any $\omega\in\Omega^\by$, the sequence of processes $\big(X^{\by,R}(\omega),\omega\in\Omega^\by\big)_{R\in\N^*}$ converges to a limit process that solves~\eqref{eq:SDE2infty} with initial condition $\by$.
\end{proposition}

\begin{proposition}\label{prop:FullMeasure}
Any Gibbs measure $\mu\in\Gcal_{\zgrande,\zpiccola}$ has its support in $\left\{ \by\in\D \text{ s.t. } P(\Omega^\by)=1 \right\}$, that is $\int_\D P((\Omega^\by)) \mu(d\by) = 1$.
\end{proposition}

\subsection{Proof of Proposition~\ref{prop:Convergence_y}: convergence on a subset}\label{sec:ball-separation}%%%%%%%%%%%%%%%

\begin{definition} 
For each configuration $\bx\in\D$, each radius $\rho>0$, and each leeway $\eps>0$:
\begin{itemize}
    \item The set of spheres which belong to $B(0,\rho)$ or to an $\eps$-chain connected to $B(0,\rho)$ is:
\begin{align*}
 \Igrande&(\bx,\rho,\eps) 
 := \Big\{ i : |\grande_i|<\rho 
             \text{ or, for some } \{ \grande_{j_1},\cdots, \grande_{j_k} \} \subset \bgrande, ~~~ |\grande_{j_1}|\le\rho, \\
           & |\grande_{j_1}-\grande_{j_2}|<2\rgrande+\eps,\ldots ,~
             |\grande_{j_{k-1}}-\grande_{j_k}|<2\rgrande+\eps  ,~
             |\grande_{j_k}-\grande_i|<2\rgrande+\eps \Big\}.
\end{align*}
    \item The set of particles which either belong to $B(0,\rho)$ or are close to some $\eps$-chain connected to $B(0,\rho)$ is:
\begin{align*}
 \Ipiccola(\bx,\rho,\eps) 
 :=\Big\{ k : |\piccola_k|<\rho  \text{ or } 
          \exists i\in\Igrande(\bx,\rho,\eps) \text{ with } 
          |\piccola_k-\grande_i|\le \rdep+\eps \Big\}.
\end{align*}
\end{itemize}
\end{definition} 

The following crucial lemma, proved in Section~\ref{app:NestedInclusionProof}, states that, if a collection of paths $X$ does not contain neither a very long chain nor a very fast ball, the balls around the origin form an interaction cluster which, for a short but strictly positive amount of time, stays separate from all other balls.
More precisely, suppose that no ball entering the region $B(0,\alpha)$ travels a distance of $\epsilon/4$ in a time of less than $\delta$, and that at no time multiple of $\delta$ does any $\epsilon$-chain of $(\kappa+1)$ spheres or more touch $B(0,\alpha)$. Dividing the time interval $[0,1]$ into intervals of length $\delta$, numbered from $a=0$ to $a=1/\delta-1$ (assumed integer), we then have nested regions $B(0,\rho_a)$, with $\rho_a$ decreasing as $a$ increases, and nested collections $\Igrande(X(a\delta),\rho_a,\eps)$ and $\Ipiccola(X(a\delta),\rho_a,\eps)$ of spheres and particles, respectively, each containing the next one. The balls that, at time $t=a\delta$, are either contained in $B(0,\rho_a)$ or near an $\eps$-chain that is connected to $B(0,\rho_a)$, until $t=(a+1)\delta$, will then move very little away from this region during the next time step and only have collisions with each other, since they will remain too far away from the other balls.

\begin{lemma}\label{LemmaNestedInclusion}
Assume that for a chain size $\kappa\ge2$, a time step $\delta$ (inverse of integer), and a region size $\alpha$, the path $X:[0,1]\to\D$ has neither too large a chain nor to large an oscillation, i.e.
\begin{equation*}
    X\notin \BadPath'_{\text{Chain}}(\delta,\alpha,\kappa,\eps) ~~~\text{ and }~~~
   X\notin \BadPath'_{\text{Fast}}(\alpha,\delta,\eps/4).
\end{equation*}
Fix $\rho$ such that $\rho_0:=\rho+\frac{2\kappa}{\delta}(2\rgrande+\eps) \le \alpha$.

Then, for any $a\in\{ 0,1,2,\cdots,\frac{1}{\delta}-1 \}$, for all $t\in[a\delta,(a+1)\delta]$, and for 
$\rho_a := \rho_0-2a\kappa(2\rgrande+\eps)
         = \rho+2\Big( \frac{1}{\delta}-a \Big)\kappa(2\rgrande+\eps)$:
\begin{itemize}
\item Spheres in $\Igrande(X(a\delta),\rho_a,\eps)$ do not bump into the other balls (separation property): 
    \begin{gather}
        \forall j\notin\Igrande(X(a\delta),\rho_a,\eps),\ 
            \forall i\in\Igrande(X(a\delta),\rho_a,\eps),\quad 
             |\grandeX_j(t)-\grandeX_i(t)| >2\rgrande+\frac{\eps}{2}, \\
        \forall k\notin\Ipiccola(X(a\delta),\rho_a,\eps),\ 
             \forall i\in\Igrande(X(a\delta),\rho_a,\eps),\quad 
             |\piccolaX_k(t)-\grandeX_i(t)| >\rdep+\frac{\eps}{2} .
    \end{gather}
\item Spheres and particles not too far from the origin stay around the origin (localisation property): 
    \begin{gather}
        \forall i\in\Igrande(X(a\delta),\rho_a,\eps),\quad 
              |\grandeX_i(t)| \le \rho_a+ \kappa(2\rgrande+\eps) +\frac{\eps}{4},\\
        \forall k\in\Ipiccola(X(a\delta),\rho_a,\eps), \quad 
         |\piccolaX_k(t)| \le \rho_a + \kappa(2\rgrande+\eps) +\rdep+\frac{5\eps}{4}.
    \end{gather}
\item The index sets form a decreasing sequence (nested inclusion property):
    \begin{gather}
        \Igrande(X(a\delta),\rho_a,\eps)
              ~\supset~ \Igrande(X((a+1)\delta),\rho_{a+1},\eps),\\
        \Ipiccola(X(a\delta),\rho_a,\eps)
              ~\supset~ \Ipiccola(X((a+1)\delta),\rho_{a+1},\eps).
    \end{gather}
\end{itemize}
\end{lemma} 

\begin{remark}
    This lemma will be used to show the convergence of the penalised trajectories $X^{\by,R}(\omega)$; we explain here the strategy.

    By definition, since $\omega\in\Omega^\by$, for $R$ large enough, the successive penalised trajectories $X^{\by,R}(\omega)$ and $X^{\by,R+1}(\omega)$ do not have very long chains nor too fast balls. We can then apply Lemma~\ref{LemmaNestedInclusion} to find that their balls starting in a very large region $B(0,\rho_0)$ have no collisions with balls starting further away. Furthermore, $X^{\by,R}(\omega)$ and $X^{\by,R+1}(\omega)$ are, locally, solutions of the same finite-dimensional SDE on the time interval $[0,1/\delta]$, since, for $R\gg\rho_0$, their penalisation drifts ($\nabla\psiGyR$, $\nabla\psiG^\by_{R+1}$, $\nabla\psiPyR$, and $\nabla\psiP^\by_{R+1}$) are equal to zero on $B(0,\rho_0)$. This means that the spheres and particles starting in $B(0,\rho_0)$ have the same paths (starting from $\by$ at time $0$) until time $t=1/\delta$, in both $X^{\by,R}(\omega)$ and $X^{\by,R+1}(\omega)$. The same reasoning applies on $[1/\delta,2/\delta]$ on a smaller region $B(0,\rho_1)$, using the fact that $X^{\by,R}(\omega,t)$ and $X^{\by,R+1}(\omega,t)$ coincide at $t=1/\delta$ in $B(0,\rho_1)$, since the balls therein are those that were in $B(0,\rho_0)$ at time $t=0$. This can be iterated for $1/\delta$ time steps, to find that the trajectories of the balls starting at most at a distance $\rho_{1/\delta}$ (given by subtracting $(1/\delta)$-times the maximal length of chains from $\rho_0$) from the origin coincide on the whole time interval $[0,1]$.
    Since $\rho_0$ can be made arbitrarily large, this means that, for all balls, there exists a value of $R$ beyond which the penalisation no longer affects them nor any other ball with which they interact. For any sphere $i$, then, the trajectories $\big(\Grande^{\by,R}_i(\omega,\cdot)\big)_R$, coincide from some $R$ large enough on, and similarly for the particles. As a consequence, the sequence of paths $\big(X^{\by,R}(\omega)\big)_R$ converges.
\end{remark}

\bigskip 

We now formalise the above reasoning.
Fix a configuration $\by=\bgrandey\bpiccolay\in\D$, a radius $\rho$, and consider an $\omega\in\Omega^\by$. 
By definition, for any fixed $\eps<\eps(\by)$, there exists $R(\eps)>0$ such that, for every $R\ge R(\eps)$,
\begin{align*}
    & X^{\by,R}(\omega) \notin \BadPath'_{\text{Chain}}(\delta(R),\alpha(R),\kappa(R),\eps),\ X^{\by,R}(\omega) \notin \BadPath'_{\text{Fast}}(\alpha(R),\delta(R),\eps/4) \\
    &  X^{\by,R+1}(\omega) \notin \BadPath'_{\text{Chain}}(\delta(R),\alpha(R),\kappa(R),\eps),\ X^{\by,R+1}(\omega) \notin \BadPath'_{\text{Fast}}(\alpha(R),\delta(R),\eps/4).
\end{align*} 
Recall that $\delta(R):=1/\floor{R^{1/3}}$, $\alpha(R):=R-2\rgrande$ and $\kappa(R):=\floor{R^{1/3}}$.
Since $\kappa(R)/\delta(R) \sim R^{2/3} \ll \alpha(R) \sim R$, 
there exists $R(\eps,\rho) \ge R(\eps)$ such that for any $R\ge R(\eps,\rho)$:
\begin{equation}\label{eq:ControlOnTheRhos}
    \rho+\frac{\eps}{\delta(R)} +\left( \frac{2}{\delta(R)}+1\right) \kappa(R)(2\rgrande+\eps) +\rdep+\frac{5\eps}{4}
    \le \alpha(R).    
\end{equation}
We now fix $R\ge R(\eps,\rho)$.
Let us prove by induction that, for all $t\in[0,1]$: for all $i$ such that $|\grandey_i|<\rho$, all $k$ such that $|\piccolay_k|<\rho$, and all $R\ge R(\eps,\rho)$,
\begin{equation}\label{eq:eventuallyconstant}
   \grandeX_i^{\by,R}(\omega,t)=\grandeX_i^{\by,R+1}(\omega,t) \text{ and }
   \piccolaX_k^{\by,R}(\omega,t)=\piccolaX_k^{\by,R+1}(\omega,t).
\end{equation}
At time $t=0$, the initial configurations $X^{\by,R}(\omega,0)=\by_{B(0,R)}$ and $X^{\by,R+1}(\omega,0)=\by_{B(0,R+1)}$ coincide as far as their spheres in\\ 
$\Igrande(X^{\by,R}(\omega,0),\rho_0,\eps)$ and their particles in 
$\Ipiccola(X^{\by,R}(\omega,0),\rho_0,\eps)$ are concerned, for
\begin{equation*}
    \rho_0 := \rho+\frac{\eps}{\delta(R)}+\frac{2\kappa(R)}{\delta(R)} (2\rgrande+\eps), 
\end{equation*}
since $\rho_0+\kappa(R)(2\rgrande+\eps) +\rdep \le R$ thanks to \eqref{eq:ControlOnTheRhos}.
This provides the initial step for the induction procedure.

We turn to the induction  step.
% From time period $[0,a\delta(R)]$ to time period $[0,(a+1)\delta(R)]$.  %%%%%%%%%%%%%
Assume that, for a fixed integer $a$ between $0$ and $1/\delta(R)$, and for 
\begin{equation}
    \rho_a := \rho+\frac{\eps}{\delta(R)}+2\Big( \frac{1}{\delta(R)}-a \Big) \kappa(R)(2\rgrande+\eps),
\end{equation}the processes $X^{\by,R}(\omega,\cdot)$ and $X^{\by,R+1}(\omega,\cdot)$ and their local times coincide on $[0,a\delta(R)]$ as far as their spheres in 
$\Igrande(X^{\by,R}(\omega,a\delta),\rho_a,\eps)$ and their particles in 
$\Ipiccola(X^{\by,R}(\omega,a\delta),\rho_a,\eps)$ are concerned, and satisfy~\eqref{eq:SDE2infty} restricted to these balls up to time $a\delta(R)$.
We want to prove the same result for $(a+1)$ instead of $a$.
The separation property in Lemma~\ref{LemmaNestedInclusion} implies that the spheres in
\begin{equation*}
    \Igrande:=\Igrande(X^{\by,R}(\omega,a\delta(R)),\rho_a,\eps)
          =\Igrande(X^{\by,R+1}(\omega,a\delta(R)),\rho_a,\eps)
\end{equation*}
or the particles in
\begin{equation*}
    \Ipiccola:=\Ipiccola(X^{\by,R}(\omega,a\delta(R)),\rho_a,\eps)
              =\Ipiccola(X^{\by,R+1}(\omega,a\delta(R)),\rho_a,\eps)  
\end{equation*}
do not bump, from time $a\delta(R)$ to time $(a+1)\delta(R)$, into the spheres and particles which do not belong to these index sets.
Thanks to ($\mathcal{S}^\by_{R}$) and its Markov property, this implies:
\small
\begin{equation*} 
\begin{cases}
    \text{for any } i\in\Igrande, \, 
    \text{ any } k\in\Ipiccola, \, \text{ and any } t\in[a\delta(R),(a+1)\delta(R)], \\ \disp  
  \Grande_i^{\by,R}(\omega,t) 
    = \Grande_i^{\by,R}(\omega,a\delta(R)) +\WGrande_i(\omega,t) -\WGrande_i(\omega,a\delta(R))
    -\frac12 \int_{a\delta(R)}^t \nabla\psiGyR\big(\Grande_i^{\by,R}(\omega,s)\big) \,ds \\ \disp \qquad\qquad\qquad
   +\sum_{j\in\Igrande} \int_{a\delta(R)}^t \big(\Grande_i^{\by,R}-\Grande_{j}^{\by,R}\big)(\omega,s) dL_{ij}(\omega,s) 
   +\sum_{k\in\Ipiccola} \int_{a\delta(R)}^t \big(\Grande_i^{\by,R}-\Piccola_{k}^{\by,R}\big)(\omega,s) d\ell_{ik}(\omega,s) \\ \disp 
 \Piccola_k^{\by,R}(\omega,t) 
 = \Piccola_k^{\by,R}(\omega,a\delta(R)) +\sigmaP\, \WPiccola_k^{\by,R}(\omega,t) -\sigmaP\, \WPiccola_k^{\by,R}(\omega,a\delta(R)) \\
    \disp \qquad\qquad\qquad
   -\frac{\sigmaP^2}{2} \int_{a\delta(R)}^t \nabla\psiPyR\big(\Piccola_k^{\by,R}(\omega,s)\big) \,ds
   +\sigmaP^2 \sum_{j\in\Igrande} \int_{a\delta(R)}^t \big(\Piccola_k^{\by,R}-\Grande_{i}^{\by,R}\big)(\omega,s)  d\ell_{ik}(\omega,s) .
\end{cases}
\end{equation*}
\normalsize
This is also true for $R+1$ instead of $R$.
The $\Grande_i^{\by,R}(\omega,\cdot)$ and $\Grande_i^{\by,R+1}(\omega,\cdot)$ for $i\in\Igrande$, as well as 
the $\Piccola_k^{\by,R}(\omega,\cdot)$ and $\Piccola_k^{\by,R+1}(\omega,\cdot)$ for $k\in\Ipiccola$,
stay in the ball $B(0,R-2\rgrande)$ where $\nabla\psiGyR$, $\nabla\psiG^\by_{R+1}$, $\nabla\psiPyR$ and $\nabla\psiP^\by_{R+1}$ vanish, as a consequence of the localisation property from Lemma~\ref{LemmaNestedInclusion}
and the inequality 
$\rho_a + \kappa(2\rgrande+\eps) +\rdep+\frac{5\eps}{4} \le R-2\rgrande$.
The restriction to these balls and this time interval of~($\mathcal{S}^\by_{R}$) and~($\mathcal{S}^\by_{R+1}$) both coincide with the corresponding restriction of~\eqref{eq:SDE2infty}.

The equality of $X^{\by,R}(\omega,\cdot)$ and $X^{\by,R+1}(\omega,\cdot)$ at time $a\delta(R)$ for spheres and particles in $\Igrande$ and $\Ipiccola$, 
and the uniqueness of the solution of the above equation imply that $X^{\by,R}(\omega,\cdot)$ and $X^{\by,R+1}(\omega,\cdot)$ -- and the corresponding local times -- coincide on the time interval $[a\delta(R),(a+1)\delta(R)]$ as far as their spheres in $\Igrande$ and their particles in $\Ipiccola$ are concerned.

In particular, 
$\Grande_i^{\by,R}(\omega,(a+1)\delta(R)) = \Grande_i^{\by,R+1}(\omega,(a+1)\delta(R))$ and 
$\Piccola_k^{\by,R}(\omega,(a+1)\delta(R)) = \Piccola_k^{\by,R+1}(\omega,(a+1)\delta(R))$
for spheres $i\in\Igrande$ and particles $k\in\Ipiccola$.
The inclusion property from Lemma~\ref{LemmaNestedInclusion} then yields
\begin{align*}
    \Igrande(X^{\by,R}((a+1)\delta(R)),\rho_{a+1},\eps) 
    = \Igrande(X^{\by,R+1}((a+1)\delta(R)),\rho_{a+1},\eps) 
    \subset
    \Igrande, \\
    \Ipiccola(X^{\by,R}((a+1)\delta(R)),\rho_{a+1},\eps) 
    = \Ipiccola(X^{\by,R+1}((a+1)\delta(R)),\rho_{a+1},\eps) 
    \subset
    \Ipiccola.
\end{align*}
As a result, $X^{\by,R}(\omega,\cdot)$ and $X^{\by,R+1}(\omega,\cdot)$ and their local times coincide on $[0,(a+1)\delta(R)]$ as far as their spheres in 
$\Igrande(X^{\by,R}(\omega,(a+1)\delta(R)),\rho_{a+1},\eps)$ and their particles in 
$\Ipiccola(X^{\by,R}(\omega,(a+1)\delta(R)),\rho_{a+1},\eps)$ are concerned, 
and the equation ($\mathcal{S}^\by_{R}$) they satisfy reduces to~\eqref{eq:SDE2infty} for these balls on this time interval.
This concludes the induction step, so that the above step-by-step equality result holds for each integer $a$ up to $1/\delta(R)$. 

The solutions $X^{\by,R}(\omega,\cdot)$ and $X^{\by,R+1}(\omega,\cdot)$ 
of the penalised equations starting from configuration $\by$ are equal up to time $1$, as it comes to their components which are spheres in $\Igrande(X^{\by,R}(\omega,1),\rho_{1/\delta(R)},\eps)$ or particles in $\Ipiccola(X^{\by,R}(\omega,1),\rho_{1/\delta(R)},\eps)$.
Moreover, since $X^{\by,R}(\omega,\cdot)$ and $X^{\by,R+1}(\omega,\cdot)$ both do not have balls that cover a distance larger than $\frac{\eps}{\delta(R)}$ during the time interval $[0,1]$ if they ever enter $B(0,\alpha(R))$, hence stay in $B(0,\rho+\frac{\eps}{\delta(R)})=B(0,\rho_{1/\delta(R)})$ if they start in $B(0,\rho)$, we have
\begin{align*}
   &\Big\{ i : |\grandey_i|<\rho  \Big\} 
   ~\subset~ \Igrande(X^{\by,R}(\omega,1),\rho_{1/\delta(R)},\eps) 
   ~\text{ and }~\\
   &\Big\{ k : |\piccolay_k|<\rho  \Big\} 
   ~\subset~ \Ipiccola(X^{\by,R}(\omega,1),\rho_{1/\delta(R)},\eps).
\end{align*} 
The penalised processes $X^{\by,R}(\omega,\cdot)$ and $X^{\by,R+1}(\omega,\cdot)$ are then the same process (including local times) for all spheres and particles starting in $B(0,\rho)$, and satisfy~\eqref{eq:SDE2infty}. 
\medbreak

Let us summarise the result of this induction process:
For any $\by=\bgrandey\bpiccolay\in\D$ and $\omega\in\Omega^\by$, 
for $\eps<\eps(\by)$, for any finite collection of spheres $\grandey_i$ and particles $\piccolay_k$, 
there exists a radius $\rho$ larger than all the $|\grandey_i|$'s and $|\piccolay_k|$'s, and there exists $R(\eps,\rho)$ large enough such that, as soon as $R>R(\eps,\rho)$, all the $\Grande^{\by,R}_i(\omega,t)$ and $\Piccola^{\by,R}_k(\omega,t)$ are equal for $t\in[0,1]$.
All the local times for the sphere-sphere collisions and the sphere-particle collisions in this collection also do not depend on $R$ as soon as $R>R(\eps,\rho)$.
Moreover, the eventually constant limit paths $\Grande^{\by,\infty}_i(\omega,\cdot):=\lim_{R\to+\infty}\Grande^{\by,R}_i(\omega,\cdot)$ and $\Piccola^{\by,\infty}_k(\omega,\cdot):=\lim_{R\to+\infty}\Piccola^{\by,R}_k(\omega,\cdot)$ satisfy~\eqref{eq:SDE2infty}.
Since $\bX^{\by,R}(\omega,0)=\by$ for all $R$'s by construction, $\bX^{\by,\infty}(\omega,0):=\by$. This concludes the proof of Proposition~\ref{prop:Convergence_y}.

\subsection{Proof of Proposition~\ref{prop:FullMeasure}: almost sure convergence} \label{sec:almost-sure}%%%%%%%%%%%%%%%%%%%%%%%%%%%%%%%%%%%%%%

We prove here Proposition~\ref{prop:FullMeasure}, that is that $ \int_\D P((\Omega^\by)^c) \mu(d\by) = 0 $ for $\mu\in\Gcal_{\zgrande,\zpiccola}$. 
Recall that, for $\delta(R):=1/\floor{R^{1/3}}$, $\alpha(R):=R-2\rgrande$ and $\kappa(R):=\floor{R^{1/3}}$, the complement set of $\Omega^\by$ is given by
\begin{align}
    (\Omega^\by)^c = \bigcap_{\eps(\by)\in1/\N} 
      ~\bigcup_{\substack{\eps\in1/\N \\ \eps<\eps(\by)}} 
      ~\limsup_{R\to+\infty} ~ \Big\{ X^{\by,R} \in \BadPath'(\eps,R) \text{ or } 
                                      X^{\by,R+1} \in \BadPath'(\eps,R) \Big\},
\end{align}           
where
$\disp \BadPath'(\eps,R) 
       := \BadPath'_{\text{Chain}}(\delta(R),\alpha(R),\kappa(R),\eps) 
          \cup \BadPath'_{\text{Fast}}(\alpha(R),\delta(R),\eps/4).$
This can be rewritten as          
\begin{align}\label{eq:OmegaComplementary}
   (\Omega^\by)^c 
 & = \bigcap_{\eps(\by)\in1/\N} 
      ~\bigcup_{\substack{\eps\in1/\N \\ \eps<\eps(\by)}} 
      ~\limsup_{R\to+\infty} ~ 
      \Big\{ X^{\by,R} \in \BadPath'(\eps,R)\cup\BadPath'(\eps,R-1) \Big\} .                
\end{align} 
Thanks to the Borel--Cantelli lemma, $\int_\D P((\Omega^\by)^c) \mu(d\by) =0$ as soon as there exists $\eps_0>0$ such that:
\begin{align}
 \forall \eps<\eps_0\,, \quad
 \sum_R \int_\D P\Big( X^{\by,R} \in \BadPath'(\eps,R)\cup\BadPath'(\eps,R-1) \Big) \mu(d\by) \ < +\infty.
\end{align} 
A sufficient condition for the above summability to hold is 
\begin{align}\label{eq:BorelCantelliQyR}
 & \exists\eps_0>0: \forall \eps<\eps_0\,, \\
   &\phantom{\exists\eps_0>0: }
   \sum_R \int_\D Q^\by_R\Big( \BadPath'(\eps,R)\cup\BadPath'(\eps,R-1) \Big) \mu(d\by) \ < +\infty, \\
 & \text{ and }\quad \sum_R \mathbf{d}_{TV}(R) ~<+\infty , \quad\text{ where } \label{eq:BorelCantelliTV} \\ 
 & \mathbf{d}_{TV}(R) 
   := \sup_{\Theta\subset \C([0,1],\D)} 
      \left| \int_\D P\Big( X^{\by,R} \in \Theta \Big) \mu(d\by)
             -\int_\D Q^\by_R\Big( \Theta \Big)\mu(d\by) \right|. 
\end{align} 
The proof of~\eqref{eq:BorelCantelliQyR} relies on the two following results on chains and fast motion, whose proofs are postponed to Sections~\ref{sec:proofLongChain} and~\ref{sec:proofFastMotion}, respectively.
\begin{lemma}\label{prop:LongChainDynamics}
For each sphere activity $\zgrande>0$, each chain length $\kappa\in\N^*$, each radius $\alpha>0$, and each chain leeway $\eps\in(0,2\rgrande)$, the probability of a long sphere chain to occur is bounded by:
\begin{align*}
   &\forall \by\in\D, \, \forall R>0\,, \quad
   Q^\by_R\Big( \BadPath'_{\text{Chain}}(\delta,\alpha,\kappa,\eps) \Big) 
   \le \frac{1}{\delta} \left( \zgrande \, ((2\rgrande+\eps)^d-(2\rgrande)^d)\volball_d \right)^\kappa.
\end{align*}
\end{lemma}

\begin{lemma}\label{prop:OmegaFast}
There exists a constant $C_{\text{Fast}}\equiv C_{\text{Fast}}(d,\sigmaP,\zgrande,\zpiccola)$, depending only on the dimension $d$, the particle diffusion coefficient $\sigmaP$, and the activities $\zgrande$ and $\zpiccola$, such that for any oscillation parameters $0<\eps\le1$, $0<\delta\le1$, and any radius $\alpha\ge1$:
\begin{align*}
 \forall \by\in\D, \ \forall R>0\,, \quad
 Q^\by_R\left( \BadPath'_{\text{Fast}}(\alpha,\delta,\eps) \right) 
 & \le C_{\text{Fast}} ~ \frac{ \alpha^d  }{\delta^{d+1}} ~
       \exp(-\frac{\eps^2}{10\, d \,\delta~\max(1,\sigmaP^2)}). 
\end{align*}
\end{lemma}
Thanks to Lemma~\ref{prop:LongChainDynamics}, 
\begin{align*} 
 & \sum_R \int_\D Q^\by_R\Bigg( \BadPath'_{\text{Chain}}(\delta(R),\alpha(R),\kappa(R),\eps) \Bigg) \mu(d\by)\\
 & \le ~ \sum_R \floor{R^{1/3}}
               \left( \zgrande \, ((2\rgrande+\eps)^d-(2\rgrande)^d)\volball_d \right)^{\floor{R^{1/3}}} ~ 
  <+\infty,
\end{align*}
as soon as $\eps$ is small enough for $\zgrande \, ((2\rgrande+\eps)^d-(2\rgrande)^d)\volball_d<1$ to hold, which is the case, for example, if 
$\eps< \min( 2\rgrande, \left( \zgrande \, (4\rgrande+1)^d \volball_d \right)^{-1} )$.
Under the same condition on $\eps$, one also obtains
\begin{align*} 
 \sum_R \int_\D Q^\by_R\Bigg( \BadPath'_{\text{Chain}}(\delta(R-1),\alpha(R-1),\kappa(R-1),\eps) \Bigg) \mu(d\by) ~ 
 <+\infty .
\end{align*}
Thanks to Lemma~\ref{prop:OmegaFast}, for any $0<\eps\le 1$,
\begin{align*} 
 & \sum_R \int_\D Q^\by_R\Bigg( \BadPath'_{\text{Fast}}(\alpha(R),\delta(R),\eps/4) \Bigg) \mu(d\by) \\ 
 & \le ~ C_{\text{Fast}} \sum_R  ~ \floor{R^{1/3}}^{d+1} ~ (R-2\rgrande)^d ~\volball_d~
                \exp\left( -\frac{ \eps^2 \floor{R^{1/3}} }{10\, d ~\max(1,\sigmaP^2)} \right) 
 <+\infty,
\end{align*}
and the same summability holds for
\begin{align*} 
  \sum_R \int_\D Q^\by_R\Bigg( \BadPath'_{\text{Fast}}(\alpha(R-1),\delta(R-1),\eps/4) \Bigg) \mu(d\by) <+\infty.
\end{align*}
Thus \eqref{eq:BorelCantelliQyR} holds.
It only remains to prove \eqref{eq:BorelCantelliTV}, i.e.\ that the total variation distance between $\int_\D Q^\by_R(\cdot) \mu(d\by)$ and $\int_\D P\big( X^{\by,R} \in~\cdot~ \big) \mu(d\by)$ converges to zero in a summable way.
From~\eqref{eq:XyR} we have that, for any measurable set $\Theta$ of continuous paths in $\D$:
\begin{equation*}
\begin{split}
    &P\Big( X^{\by,R} \in \Theta \Big) = P\Big( \bX^{\by,R,\sharp(\bgrandey_{B(0,R)}),\sharp(\bpiccolay_{B(0,R)})}(\by_{B(0,R)},~\cdot~) \in \Theta \Big) 
       =: F^\by_\Theta(\by_{B(0,R)}) ,
\end{split}
\end{equation*}
where
$ F^\by_\Theta(\grande\piccola) 
       := P\big( \bX^{\by,R,\sharp\grande,\sharp\piccola}(\grande\piccola,~\cdot~) \in \Theta \big)$. Note that $F^\by_\Theta$ only depends on $\by_{{B(0,R)}^c}$ and not on  $\by_{B(0,R)}$.
Taking $\Lambda=B(0,R)$, with $\bgrande=\{\grande_1, \dots, \grande_\ngrande\}$ and $\bpiccola=\{\piccola_1, \dots, \grande_\npiccola\}$, in the definition~\eqref{eq:mu} of $\mu$, we find:
\begin{align*} 
& \int_\D P\Big( X^{\by,R} \in \Theta \Big) \, \mu(d\by) 
  = \int_\D F^\by_\Theta(\by_{B(0,R)}) \, \mu(d\by) \\
& = \int_\D \frac{1}{Z_{B(0,R)}(\by)} 
    \sum_{\ngrande=0}^{\infty} \frac{\zgrande^\ngrande}{\ngrande !}
    \sum_{\npiccola=0}^{\infty} \frac{\zpiccola^\npiccola}{\npiccola !}\\
&\phantom{= \int}
    \int_{B(0,R)^\ngrande} \int_{B(0,R)^\npiccola} 
    F^\by_\Theta(\bgrande\bpiccola) \, 
    \un_{\D}(\by_{B(0,R)^c} \bgrande\bpiccola) 
    \otimes_{k=1}^\npiccola d\piccola_k \otimes_{i=1}^\ngrande d\grande_i \, \mu(d\by).
\end{align*}
Recall that $\BallGyR$ and $\BallPyR$ are the parts of $B(0,R)$ where a sphere or a particle can be put without conflict with the outside configuration $\by_{B(0,R)}$ (see~\eqref{eq:BallGyR_BallPyR}):
\begin{align*} 
& \int_\D P\Big( X^{\by,R} \in \Theta \Big) \, \mu(d\by) \\
& = \int_\D \frac{1}{Z_{B(0,R)}(\by)} 
    \sum_{\ngrande=0}^{\infty} \frac{\zgrande^\ngrande}{\ngrande !}
    \sum_{\npiccola=0}^{\infty} \frac{\zpiccola^\npiccola}{\npiccola !}
    \int_{(\BallGyR)^\ngrande} \int_{(\BallPyR)^\npiccola} 
    F^\by_\Theta(\bgrande\bpiccola) \, \un_{\D}(\bgrande\bpiccola) 
    \otimes_{k=1}^\npiccola d\piccola_k \otimes_{i=1}^\ngrande d\grande_i \, \mu(d\by).
\end{align*}
According to the definitions~\eqref{eq:QRy},~\eqref{eq:Q_yRn}, and~\eqref{eq:nu_yRn} of the measures $Q^\by_R$, $Q^\by_{R,\ngrande,\npiccola}$, and $\nu^\by_{R,\ngrande,\npiccola}$, we have
\begin{align*}
    Q^\by_R\Big( \Theta \Big) 
    &= \frac{1}{\Zmixing^\by_R} 
       \sum_{\ngrande=0}^{+\infty} \frac{\zgrande^\ngrande}{\ngrande!} \, 
       \sum_{\npiccola=0}^{+\infty} \frac{\zpiccola^\npiccola}{\npiccola!} \, 
       \int P\left( \bX^{\by,R,\ngrande,\npiccola}(\bx,\cdot) \in\Theta \right) \, 
       d\nu^\by_{R,\ngrande,\npiccola}(\bx) \\
    & = \frac{1}{\Zmixing^\by_R} 
       \sum_{\ngrande=0}^{+\infty} \frac{\zgrande^\ngrande}{\ngrande!} \, 
       \sum_{\npiccola=0}^{+\infty} \frac{\zpiccola^\npiccola}{\npiccola!} 
       \int_{\R^{\ngrande d}} \int_{\R^{\npiccola d}} F^\by_\Theta(\bgrande\bpiccola) \, \un_{\D}(\bgrande\bpiccola) \, 
       \otimes_{i=1}^\ngrande \lambdaGyR(d\grande_i) \otimes_{k=1}^\npiccola \lambdaPyR(d\piccola_k).
\end{align*}
Distributing the terms according to the number of spheres and particles in $B(0,R)$, we get the total variation distance:
\begin{align*} 
   &\mathbf{d}_{TV}(R) \\
   &= \sup_{\Theta} \Bigg| \int_\D  
     \sum_{\ngrande=0}^{\infty} \frac{\zgrande^\ngrande}{\ngrande !}
     \sum_{\npiccola=0}^{\infty} \frac{\zpiccola^\npiccola}{\npiccola !}
     \int_{\R^{\ngrande d}} \int_{\R^{\npiccola d}}
     F^\by_\Theta(\bgrande\bpiccola) \, \un_{\D}(\bgrande\bpiccola) \, \Delta(\by,\bgrande\bpiccola) 
     \otimes_{k=1}^\npiccola d\piccola_k \otimes_{i=1}^\ngrande d\grande_i \mu(d\by) \Bigg|,
\end{align*} 
where
\begin{align*}
   \Delta(\by,\bgrande\bpiccola) 
   := \frac{ \prod_{i=1}^\ngrande \un_{\grande_i\in\BallGyR} \, 
             \prod_{k=1}^\npiccola \un_{\piccola_i\in\BallPyR} }{Z_{B(0,R)}(\by)} 
      -\frac{ e^{-\sum_{i=1}^\ngrande \psiGyR(\grande_i) } \, 
              e^{-\sum_{k=1}^\npiccola \psiPyR(\piccola_k)} }{\Zmixing^\by_R}.
\end{align*} 
Using the fact that $|F^\by_\Theta(\bgrande\bpiccola)|$ is bounded by $1$, we get
\begin{align*} 
  &\mathbf{d}_{TV}(R)\\
  &\le \int_\D 
      \sum_{\ngrande=0}^{\infty} \frac{\zgrande^\ngrande}{\ngrande !}
      \sum_{\npiccola=0}^{\infty} \frac{\zpiccola^\npiccola}{\npiccola !}
      \int_{\R^{\ngrande d}} \int_{\R^{\npiccola d}}
      \un_{\D}(\bgrande\bpiccola) \, \abs{\Delta(\by,\bgrande\bpiccola)}
      \otimes_{k=1}^\npiccola d\piccola_k \otimes_{i=1}^\ngrande d\grande_i \, \mu(d\by).
\end{align*} 
Note that the terms are well-defined even for $\ngrande=0$ and/or $\npiccola=0$, with the convention that $\prod_{i=1}^0 \cdots =1$ and $\sum_{i=1}^0 \cdots =0$.
We split the difference $\Delta(\by,\bgrande\bpiccola)$ in two: 
\begin{align*} 
  &\Big| \Delta(\by,\bgrande\bpiccola) \Big|\\
  &\le \prod_{i=1}^\ngrande \un_{\grande_i\in\BallGyR} \, 
       \prod_{k=1}^\npiccola \un_{\piccola_i\in\BallPyR}
       \Bigg| \frac{1}{Z_{B(0,R)}(\by)} -\frac{1}{\Zmixing^\by_R}  \Bigg| \\
  & \quad +\frac{1}{\Zmixing^\by_R} 
           \Bigg| \prod_{i=1}^\ngrande \un_{\grande_i\in\BallGyR} \, 
                  \prod_{k=1}^\npiccola \un_{\piccola_i\in\BallPyR} 
                  - e^{-\sum_{i=1}^\ngrande \psiGyR(\grande_i) } \, 
                    e^{-\sum_{k=1}^\npiccola \psiPyR(\piccola_k)} \Bigg|.
\end{align*}
Since $\psiGyR$ vanishes on $\BallGyR$ and $\psiPyR$ vanishes on $\BallPyR$, the argument of the second absolute value is negative, since
\begin{equation*}
   \prod_{i=1}^\ngrande \un_{\grande_i\in\BallGyR} \,
   \prod_{k=1}^\npiccola \un_{\piccola_i\in\BallPyR}
   \le e^{-\sum_{i=1}^\ngrande \psiGyR(\grande_i) } \, 
             e^{-\sum_{k=1}^\npiccola \psiPyR(\piccola_k)},    
\end{equation*}
and the argument of the first absolute value in the bound of 
$\abs{\Delta(\by,\bgrande\bpiccola)}$ is positive, since
\begin{align*}
&Z_{B(0,R)}(\by)\\
& = \sum_{\ngrande=0}^{+\infty} \frac{\zgrande^\ngrande}{\ngrande!} \,                
    \sum_{\npiccola=0}^{+\infty} \frac{\zpiccola^\npiccola}{\npiccola!} 
    \int_{\R^{\ngrande d}} \int_{\R^{\npiccola d}}
    \prod_{i=1}^\ngrande \un_{\grande_i\in\BallGyR} \,
    \prod_{k=1}^\npiccola \un_{\piccola_i\in\BallPyR} \,
    \un_{\D}(\bgrande\bpiccola) \, 
    \otimes_{i=1}^\ngrande d\grande_i \otimes_{k=1}^\npiccola d\piccola_k \\
&\le \sum_{\ngrande=0}^{+\infty} \frac{\zgrande^\ngrande}{\ngrande!} \,                
     \sum_{\npiccola=0}^{+\infty} \frac{\zpiccola^\npiccola}{\npiccola!} \,
     \int_{\R^{\ngrande d}} \int_{\R^{\npiccola d}} \un_{\D}(\bgrande\bpiccola) \, 
     \otimes_{i=1}^\ngrande \lambdaGyR(d\grande_i) \otimes_{k=1}^\npiccola \lambdaPyR(d\piccola_k) 
     = \Zmixing^\by_R.
\end{align*} 
Consequently, the first part in the upper bound of $\abs{\Delta(\by,\bgrande\bpiccola)}$ yields $Z_{B(0,R)}(\by)$ and the second part yields $\Zmixing^\by_R - Z_{B(0,R)}(\by)$\,:
\begin{align} 
& \mathbf{d}_{TV}(R) \\
% &\le \int_\D  
%      \sum_{\ngrande=0}^{+\infty} \frac{\zgrande^\ngrande}{\ngrande!} \, 
%      \sum_{\npiccola=0}^{+\infty} \frac{\zpiccola^\npiccola}{\npiccola!} \\
% &\phantom{\le \int_\D}
%      \int_{(\BallGyR)^\ngrande} \int_{(\BallPyR)^\npiccola} \un_{\D}(\bgrande\bpiccola) 
%      \Bigg( \frac{1}{Z_{B(0,R)}(\by)} -\frac{1}{\Zmixing^\by_R} \Bigg) 
%      \otimes_{k=1}^\npiccola d\piccola_k \otimes_{i=1}^\ngrande d\grande_i \,  \\    
% &\phantom{\leq}  
%      +\frac{1}{\Zmixing^\by_R} \int_{\R^{\ngrande d}} \int_{\R^{\npiccola d}}
%       \un_{\D}(\bgrande\bpiccola) 
%       \Bigg( e^{-\sum_{i=1}^\ngrande \psiGyR(\grande_i) } \, 
%              e^{-\sum_{k=1}^\npiccola \psiPyR(\piccola_k)}\\
% &\phantom{\leq +\frac{1}{\Zmixing^\by_R} \int_{\R^{\ngrande d}} \int_{\R^{\npiccola d}}\un_{\D}(\bgrande\bpiccola) \Bigg(\quad}
%              -\prod_{i=1}^\ngrande \un_{\grande_i\in\BallGyR} \,
%              \prod_{k=1}^\npiccola \un_{\piccola_i\in\BallPyR} \Bigg)\\
%     & \phantom{\leq +\frac{1}{\Zmixing^\by_R} \int_{\R^{\ngrande d}} \int_{\R^{\npiccola d}}\un_{\D}(\bgrande\bpiccola) \Bigg(\qquad\qquad}
%     \otimes_{k=1}^\npiccola d\piccola_k \otimes_{i=1}^\ngrande d\grande_i \, 
%     \mu(d\by) \\
&\le \int_\D  
     \Bigg( \frac{1}{Z_{B(0,R)}(\by)} -\frac{1}{\Zmixing^\by_R} \Bigg)       Z_{B(0,R)}(\by)
     +\frac{1}{\Zmixing^\by_R} 
      \Bigg( \Zmixing^\by_R - Z_{B(0,R)}(\by) \Bigg) \, 
    \mu(d\by) \nonumber \\ 
&\le 2\int_\D  
     \frac{ \Zmixing^\by_R - Z_{B(0,R)}(\by) }{\Zmixing^\by_R} \, \mu(d\by).\label{eq:dTVR}
\end{align}
To prove \eqref{eq:BorelCantelliTV}, we only have to find a summable upper bound for the above integral. 
Since $\psiGyR(\grande_i)=0$ for $\grande_i\in\BallGyR$\,, and $\psiPyR(\piccola_k)=0$ for $\piccola_k\in\BallPyR$\,, we have
\begin{equation*}
\begin{split}
    &1 - e^{\sum_{i=1}^\ngrande \psiGyR(\grande_i) } \, 
       e^{\sum_{k=1}^\npiccola \psiPyR(\piccola_k)}
       \prod_{i=1}^\ngrande \un_{\grande_i\in\BallGyR} \,
       \prod_{k=1}^\npiccola \un_{\piccola_i\in\BallPyR}\le \sum_{i=1}^\ngrande \un_{\grande_i\notin\BallGyR} 
             +\sum_{k=1}^\npiccola \un_{\piccola_i\notin\BallPyR}.  
\end{split} 
\end{equation*}
Using the exchangeability of the spheres and particles, we obtain
\begin{align*} 
 & \Zmixing^\by_R - Z_{B(0,R)}(\by) \\
 & \le \sum_{\ngrande=0}^{+\infty} \frac{\zgrande^\ngrande}{\ngrande!} \,                
       \sum_{\npiccola=0}^{+\infty} \frac{\zpiccola^\npiccola}{\npiccola!} \,
       \int_{\R^{\ngrande d}} \int_{\R^{\npiccola d}} \un_{\D}(\bgrande\bpiccola) \, 
       \Bigg( \ngrande \un_{\grande_\ngrande \notin\BallGyR} 
             +\npiccola \un_{\piccola_\npiccola \notin\BallPyR} \Bigg)
       \otimes_{i=1}^\ngrande \lambdaGyR(d\grande_i) \otimes_{k=1}^\npiccola \lambdaPyR(d\piccola_k).
\end{align*}       
Since $\un_{\D}(\grande_1 \dots \grande_\ngrande ~ \piccola_1 \dots \piccola_\npiccola)
 \le \un_{\D}(\grande_1 \dots \grande_{\ngrande-1} ~ \piccola_1 \dots \piccola_\npiccola)$, and analogously for the particles,
\begin{align*} 
 & \Zmixing^\by_R - Z_{B(0,R)}(\by) \\  
 & \le \zgrande \sum_{\ngrande=1}^{+\infty} \frac{\zgrande^{\ngrande-1}}{(\ngrande-1)!} \,   
                \sum_{\npiccola=0}^{+\infty} \frac{\zpiccola^\npiccola}{\npiccola!} \,
       \int_{\R^{(\ngrande-1) d}} \int_{\R^{\npiccola d}} \un_{\D}(\bgrande\bpiccola) 
       \otimes_{i=1}^{\ngrande-1} \lambdaGyR(d\grande_i) 
       \otimes_{k=1}^\npiccola \lambdaPyR(d\piccola_k) ~
       \lambdaGyR((\BallGyR)^c) \\
 &\quad   
       + \zpiccola \sum_{\ngrande=0}^{+\infty} \frac{\zgrande^\ngrande}{\ngrande!} \,                
                   \sum_{\npiccola=1}^{+\infty} \frac{\zpiccola^{\npiccola-1}}{(\npiccola-1)!} \,
         \int_{\R^{\ngrande d}} \int_{\R^{(\npiccola-1) d}} \un_{\D}(\bgrande\bpiccola)        
         \otimes_{i=1}^\ngrande \lambdaGyR(d\grande_i) 
         \otimes_{k=1}^{\npiccola-1} \lambdaPyR(d\piccola_k) ~
         \lambdaPyR((\BallPyR)^c) \\  
 & \le \Zmixing^\by_R \, \Big( \zgrande \lambdaGyR((\BallGyR)^c) 
                               + \zpiccola \lambdaPyR((\BallPyR)^c) \Big).
\end{align*}
Injecting this in~\eqref{eq:dTVR} yields
\begin{align}\label{eq:TotalVariationUpperBound}
\mathbf{d}_{TV}(R)  
&\le 2\int_\D \zgrande \lambdaGyR((\BallGyR)^c) 
              + \zpiccola \lambdaPyR((\BallPyR)^c) \, \mu(d\by).
\end{align}
Finally, by assumption~\eqref{eq:psiGR}, we have that
$\sum_R \mathbf{d}_{TV}(R) <+\infty$,
which in turn implies that $ \int_\D P((\Omega^\by)^c) \mu(d\by) = 0 $, concluding the proof of Proposition~\ref{prop:FullMeasure}.

\medskip

As a result, the solution $X^{\by,\infty}(\omega,\cdot)$ of~\eqref{eq:SDE2infty} constructed in Section \ref{sec:ball-separation} as a limit process exists for $\mu$-almost every $\omega$, for any fixed $\by$ in the interior of $\D$.

\begin{remark}\label{rmk:uniqueness}
The process $X^{\by,\infty}(\omega,\cdot)$ is the unique solution of~\eqref{eq:SDE2infty} in the sense of~\cite[Lemma 5.4]{TanemuraEDS}. 
This means that $X^{\by,\infty}$ coincides $\mu$-a.s. with any solution $X$ of~\eqref{eq:SDE2infty} starting from $\by$ in the class of paths for which:
there exists $\eps>0$ and a rate $p\in\N^*$ such that,
for all $\rho>0$ and infinitely many $m\in\N^*$, 
there exists a finite sequence of rational times $t_0=0<t_1<\cdots<t_{m'-1}<t_{m'}=1$, with $t_{a+1}-t_a \le \frac{1}{m}$, and a sequence of bounded sets $C_0,\ldots C_{m'-1}$ in $\R^d$, with $B(0,\rho+m)\subset C_{m'-1}\subset\dots\subset C_0\subset B(0,\rho+m+m^p)$, with $B(C_{a+1},\eps) \subset C_a$ for each $a$, such that, for every $a\in\{0,\ldots,m'-1\}$, the boundary of $C_a$ separates balls in such a way that they do not interact:
\begin{align*} 
    &\forall a\in\{0,1,\ldots,m'-1\}, \quad 
    \forall t\in\left[\frac{a}{m},\frac{a+1}{m}\right ], \\
    &\qquad \inf\{|\Grande_i(t)-x|: ~i\in\N^*,~ x\in\partial C_a \} > \rgrande+\frac{\eps}{4}, \\
    &\qquad \text{and }  
    \inf\{|\Piccola_k(t)-x|: ~k\in\N^*,~ x\in\partial C_a \} > \rpiccola+\frac{\eps}{4} .
\end{align*}   
That the limit solution $X^{\by,\infty}$ belongs to this class follows from the computations of Section~\ref{sec:ball-separation}. 
Any solution belonging to this class inherits the uniqueness property of the finite SDE, since every sphere or particle belongs on successive time intervals to finite sets of interacting balls, hence they coincide with $X^{\by,\infty}(\omega,\cdot)$.
\end{remark}

\subsection{Reversibility of the limit under Gibbs measures}\label{sec:reversibility}
%%%%%%%%%%%%%%%%%%%%%%%%%%%%%%%%%%%%%%%%%%%%%%%%%%%%%%%%%%%%%%%%%%%%%%%%%%%%

To complete the proof of Theorem~\ref{theorem:SDE2infty}, we only have to note that the limit process $X^{\by,\infty}(\omega,\cdot)$, constructed in Sections~\ref{sec:ball-separation} and~\ref{sec:almost-sure} for $\mu$-almost every $\omega$, admits $\mu$ as a reversible measure, for any Gibbs measure $\mu\in\Gcal_{\zgrande,\zpiccola}$.

This is a straightforward consequence of the reversibility of the penalised processes under the approximate measures $Q^\by_R$.
Indeed, for any time $T$ in $[0,1]$, any subdivision $0\leq t_1 < \dots <  t_j \leq T$, and any bounded continuous local functions $F_1, \dots, F_j$ on $\M$, we have
\begin{align*} 
 & \left| \int_\D E\bigg( \prod_{i=1}^j F_i\big( X^{\by,\infty}(T-t_i) \big) 
                          -\prod_{i=1}^j F_i\big( X^{\by,\infty}(t_i) \big) 
                   \bigg) \mu(d\by) \right| \\
& = \lim_{R\to+\infty} \left| \int_\D E\bigg( 
    \prod_{i=1}^j F_i\big( X^{\by,R}(T-t_i) \big) 
    -\prod_{i=1}^j F_i\big( X^{\by,R}(t_i) \big) \bigg) \mu(d\by) \right| \\
 & \le \lim_{R\to+\infty} 2 \prod_{i=1}^j \|F_j\|_\infty ~ \mathbf{d}_{TV}(R) =0,
\end{align*} 
since, thanks to the reversibility of $X^{\by,R}$ starting from $\mu^\by_R$,
\begin{equation*}
    \int_\D E\bigg( 
    \prod_{i=1}^j F_i\big( X^{\by,R}(T-t_i) \big) 
    -\prod_{i=1}^j F_i\big( X^{\by,R}(t_i) \big) \bigg) \mu^\by_R(d\by) =0. 
\end{equation*}

%%%%%%%%%%%%%%%%%%%%%%%%%%%%%%%%%%%%%%%%%%%%%%%%%%%%%%%%%%%%%%%%%%%%%%%%%%%%%%%%%%%%
\section{Occurrence of a depletion interaction between hard spheres}\label{sec:depletion}
%%%%%%%%%%%%%%%%%%%%%%%%%%%%%%%%%%%%%%%%%%%%%%%%%%%%%%%%%%%%%%%%%%%%%%%%%%%%%%%%%%%%

In this section, we study the emergence of an attractive interaction, the \emph{depletion interaction}, between hard spheres, due to the presence of the particle medium. We first identify the projection $\mugrande$ of the reversible measure $\mu$ introduced in~\eqref{eq:mu}, as a Gibbs measure with a new effective interaction. We then construct a gradient dynamics whose reversible measure is given by $\mugrande$. Finally, in Sections~\ref{sec:percolation} and~\ref{sec:packing}, we study the properties of the Gibbs measures $\mugrande$ in the low- and high-activity regimes, respectively.

\subsection{The projection of the two-type reversible measure}\label{sec:trace}
%%%%%%%%%%%%%%%%%%%%%%%%%%%%%%%%%%%%%%%%%%%%%%%%%%%%%%%%%%%%%%%%%%%%%%%%%%%%

We study here the projection of the reversible measure $\mu\in\Gcal_{\zgrande,\zpiccola}$ onto the subsystem of hard spheres.
For any finite configuration $\bgrande=\{\grande_1,\dots,\grande_n\}$, $n\geq 1$, of hard spheres, consider the energy
\begin{equation}
    \nrg (\bgrande) := \displaystyle \vol{\BGrande(\bgrande)} = \displaystyle \vol{\bigcup_{i=1}^{n}B(\grande_i,\rdep) }.
\end{equation}
We denote by $\Gcal_{\zgrande}(\zpiccola\nrg)$ the set of Gibbs measures associated to $\nrg$ with activity $\zgrande$ and inverse temperature $\zpiccola$, that is the set of measures solutions of
\begin{equation} \label{eq:mugrande}
\begin{split}
    &\int_\D F(\bgrande) \, \mugrande(d\bgrande) = \int_\D \frac{1}{\Zgrande_\Lambda(\bgrandey)} \int_{\MGrande} 
    F(\bgrandey_{\Lambda^c} \bgrande) \, 
    e^{-\zpiccola\nrg_\Lambda(\bgrande\bgrandey_{\Lambda^c})}\un_{\D}(\bgrandey_{\Lambda^c} \bgrande) \, 
    {\pi^\zgrande_\Lambda}(d\bgrande) \, \mugrande(d\bgrandey),
\end{split}
\end{equation}
for any $\Lambda\subset\R^d$ and any positive bounded measurable $F$ on $\MGrande$, where 
\begin{equation}\label{eq:conditional_en}
    \nrg_\Lambda (\bgrande) := \displaystyle \vol{\BGrande(\bgrande_\Lambda)\setminus \BGrande(\bgrande_{\Lambda^c})} = \displaystyle \vol{\bigcup_{\grande\in\bgrande_\Lambda}B(\grande,\rdep)\setminus \bigcup_{\grande\in\bgrande_{\Lambda^c}}B(\grande,\rdep)}.
\end{equation}
is the conditional energy in $\Lambda$ for (infinite) configurations of hard spheres $\bgrande\in\MGrande$.
\begin{remark}
    It is worth noting that several choices of conditional energy lead to the same set of Gibbs measures. This is the case, for example, for $\tilde\nrg_\Lambda (\bgrande) := \displaystyle \vol{\BGrande(\bgrande)\cap\Lambda}$, since it differs from $\nrg_\Lambda$ only for factors that only depend on the external configuration. As was the case for $\Gcal_{\zpiccola,\zgrande}$, existence of at least one element of $\Gcal_{\zgrande}(\zpiccola\nrg)$, for any value of $\zgrande,\zpiccola>0$, is a classical result. 
\end{remark}
Similarly to the Widom--Rowlinson model~\cite{widom1970new}, the following correspondence between two-type and one-type models holds, where the activity $\zpiccola$ of the particles plays the role of \emph{inverse temperature} for the effective interaction between hard spheres.
\begin{proposition}\label{prop:tracemesureequ}
    For any $\zgrande,\zpiccola>0$, 
    \begin{itemize}
        \item If $\mu\in\Gcal_{\zpiccola,\zgrande}$ is a Gibbs measure on $\M$, then its marginal on $\MGrande$ is a Gibbs measure $\mugrande_\zpiccola\in\Gcal_{\zgrande}(\zpiccola\nrg)$.
        \item If $\mugrande\in\Gcal_{\zgrande}(\zpiccola\nrg)$ is a Gibbs measure on $\MGrande$, then $\mugrande(d\bgrande)\otimes\pi^{\zpiccola}_{\R^d\setminus\BGrande(\bgrande)}(d\bpiccola)\in\Gcal_{\zgrande,\zpiccola}$ is a Gibbs measure on $\M$.
    \end{itemize}
\end{proposition} 
The proof is relatively straightforward, but it is still worthwhile to include here. Indeed, in the Widom--Rowlinson model this result is folklore, and as such a rigorous proof using the DLR formalism is not easily found. Moreover, as far as we are aware, this is the first time this argument is extended to this setting where there is an additional hard-core interaction. 

\begin{proof}

We first show that if $\mugrande$ satisfies the DLR equations~\eqref{eq:DLR_dep} for $\nrg$, then $\mu:=\mugrande\otimes\pi^{\zpiccola}_{\R^d\setminus\BGrande(\bgrande)}$ satisfies the DLR equations~\eqref{eq:mu}.
For any positive bounded measurable $F$ on $\D$, we have
\begin{equation*}
\begin{split}
    &\int F(\bx)\mu(d\bx):= \iint F(\bpiccola\bgrande)\, \pi^{\zpiccola}_{\R^d\setminus\BGrande(\bgrande)}(d\bpiccola)\, \mugrande(d\bgrande) \\
    &= \iiint F(\bpiccola'_\Lambda\bgrande'_\Lambda\bpiccola_{\Lambda^c}\bgrande_{\Lambda^c})\frac{e^{-\zpiccola\vol{\BGrande(\bgrande'_\Lambda)\setminus \BGrande(\bgrande_{\Lambda^c})}}}{\Zgrande_\Lambda(\bgrande_{\Lambda^c})}\pi^{\zpiccola}_{\BGrande(\bgrande'_\Lambda\bgrande_{\Lambda^c})^c}(d\bpiccola)
    \un_{\D}(\bgrande'_\Lambda\bgrande_{\Lambda^c}) \pi^{\zgrande}_\Lambda(d\bgrande'_{\Lambda})\mugrande(d\bgrande_{\Lambda^c}).
\end{split}
\end{equation*}
Writing $e^{-\vol{\BGrande(\bgrande'_\Lambda)\setminus \BGrande(\bgrande_{\Lambda^c})}} = e^{-\vol{\BGrande(\bgrande'_\Lambda)\cup \BGrande(\bgrande_{\Lambda^c})}}\, e^{ \vol{\BGrande(\bgrande_{\Lambda^c})}}$, and absorbing the second term into $Z_\Lambda(\bx_{\Lambda^c})$, we have that this is precisely the Poisson void probability for $\pi^\zpiccola_\Lambda\pi^\zpiccola_{\Lambda^c}$, so that we find
\begin{equation*}
\begin{split}
    &\int F(\bx)\mu(d\bx) = \int\frac{1}{Z_\Lambda(\bx_{\Lambda^c})}\left(\iint F(\bx'_\Lambda\bx_{\Lambda_c}) \un_\D(\bx'_{\Lambda}\bx_{\Lambda^c})\pi_\Lambda^\zpiccola(d\bpiccola')\pi_\Lambda^\zgrande(\bgrande')\right) \mu(d\bx).
\end{split}
\end{equation*} 
Conversely, let $\mu\in\Gcal_{\zpiccola,\zgrande}$. Its marginal on the hard spheres is characterised by integration over test functions $F$ supported on $\MGrande$, that is:
$\int_{\MGrande} F(\bgrande)\mugrande_\zpiccola(d\bgrande):= \int_\M F(\bgrande)\mu(d\bgrande\bpiccola)$.
The DLR equations for $\mu$ then yield
\begin{equation*}
\begin{split}
    &\int F(\bgrande)\mu(d\bgrande\bpiccola) \\
    &= \int \frac{1}{Z_\Lambda(\by)} \int_{\MGrande} \int_{\MPiccola} 
    F(\bgrandey_{\Lambda^c} \bgrande) \, 
    \un_{\D}(\bgrandey_{\Lambda^c}\bpiccolay_{\Lambda^c} \bgrande\bpiccola) \, 
    \pi^\zpiccola_{\Lambda}(d\bpiccola)  \, \pi^\zgrande_{\Lambda}(d\bgrande) \, \mu(d\by)\\
    &= \int \frac{1}{Z_\Lambda(\by)} \int_{\MGrande} \int_{\MPiccola} 
    F(\bgrandey_{\Lambda^c} \bgrande) \, 
    \un_{\D}(\by_{\Lambda^c} \bgrande\bpiccola) \, 
    \pi^\zpiccola_{\Lambda}(d\bpiccola)  \, \pi^\zgrande_{\Lambda}(d\bgrande) \, \mu(d\by)\\
    &= \int \frac{1}{\Zgrande_\Lambda(\bgrandey)} \int_{\MGrande}
    F(\bgrandey_{\Lambda^c} \bgrande) \, 
    \exp\big(-\zpiccola\nrg_\Lambda(\bgrande\bgrandey_{\Lambda^c})\big)\un_{\D}(\bgrandey_{\Lambda^c}\bgrande) \, 
    \pi^\zgrande_{\Lambda}(d\bgrande) \, \mu(d\bgrandey\bpiccolay),
\end{split}
\end{equation*}
where at each step we absorbed the factors constant in $\bgrande$ into $Z_\Lambda(\by)$, e.g.\ $e^{\zgrande\vol{\BGrande(\bpiccolay_{\Lambda^c})\cap\Lambda}}$. This finally leads to the new partition function 
\begin{equation*}
    \Zgrande_\Lambda(\bgrandey)\equiv \Zgrande_\Lambda(\bgrandey_{\Lambda^c}) =\int_{\MGrande}\exp\big(-\zpiccola\nrg_\Lambda(\bgrande\bgrandey_{\Lambda^c})\big)\un_{\D}(\bgrandey_{\Lambda^c} \bgrande) \, {\pi^\zgrande_\Lambda}(d\bgrande),
\end{equation*}
concluding the proof.
\end{proof}

\subsection{An associated gradient dynamics}\label{sec:depletionSDE}
%%%%%%%%%%%%%%%%%%%%%%%%%%%%%%%%%%%%%%%%%%%%%%%%

We are interested here in constructing a diffusive dynamics for infinitely many hard spheres which has~\eqref{eq:mugrande} as reversible measure.

It is a simple geometric fact (see, e.g.~\cite{LT11}) that, if $\rpiccola/\rgrande\leq\frac{2}{3}\sqrt{3}-1=:\rho_2 \simeq 0.15$, the interaction $\nrg$ between any finite number $n\geq 1$ of hard spheres $\bgrande=\{\grande_1,\dots,\grande_n\}$ reduces to 
\begin{equation*}
    \nrg (\bgrande) = n \volball_d \,  \rdep\rspace ^d  - \sum_{1\leq i<j\leq n}  \Vovlap\Big(\frac{|\grande_i - \grande_j|}{2 \rdep}\Big),
\end{equation*}
where, for $u= \frac{|x_i - x_j|}{ 2 \rdep} \in \Big[\frac{\rgrande}{\rdep},+\infty \Big)$,
\begin{equation*}
    \Vovlap(u) =  
    2\,\volball_{d-1} ~ \rdep\rspace ^d \int_0^{\arccos(u)}  (\sin\theta)^d \,d\theta 
    \ \un_{[\rgrande/\rdep,1]} (u) := - \phidep_2(x_{i},x_{j})
\end{equation*}
is an attractive two-body potential, which only depends on the distance between points.
It decreases from its maximal value $\Vovlap^*$, attained at $u= \frac{\rgrande}{\rdep}$, to its minimal value $0$, attained at $u=1$, and vanishes on $[1,+\infty)$.
With this notation, the conditional energy~\eqref{eq:conditional_en} can be rewritten as
\begin{equation*}
    \nrg_\Lambda (\bgrande_\Lambda\bgrandey_{\Lambda^c}) 
    = \volball_d \,  \rdep\rspace ^d ~\sharp\bgrande_\Lambda
      + \sum_{\grande_i,\grande_j}\phi_2(\grande_i,\grandey_j) 
      + \sum_{\grande_i,\grandey_j}\phi_2(\grande_i,\grandey_j).
\end{equation*}
Note that $\Vovlap$ is $\C^2$ for $d\geq 3$. Indeed, its derivatives are given, for $\frac{1}{1+\rho} <u< 1$, by
\begin{align*}
    &\Vovlap'(u) = 2\volball_{d-1} ~ \rdep\rspace ^{d} \arccos'(u) \left( \sin\big(\arccos(u)\big) \right)^d 
    = -2\volball_{d-1} ~ \rdep\rspace ^{d} \left( 1-u^2\right)^{\frac{d-1}{2}},\\
    &\Vovlap''(u) = 2(d-1) \volball_{d-1} ~ \rdep\rspace ^{d} u \left( 1-u^2\right)^{\frac{d-3}{2}} = -(d-1) ~ \frac{ u~\Vovlap'(u) }{1-u^2}.
\end{align*}
This allows us to show, using~\cite[Theorem 1.4 and Remark 1.2]{fradon_roelly_tanemura_2000},
existence and uniqueness of a strong solution to the corresponding diffusive dynamics of the infinitely many hard spheres $(\Grande_i )_{i\in\N^*}$ submitted to ($\zpiccola$-times) this gradient field:
\begin{equation}\tag{${\mathcal{S}}^\textrm{dep}$}\label{eq:Sdep_pair}
\begin{cases}
\begin{array}{l}
    \textrm{for any } i\in\N^*, \, t \in [0,1] ,                             \\ 
    d \Grande_i (t) =
    d \WGrande_i(t)\\
    \phantom{d \Grande_i (t) =}
    - \disp \frac{\zpiccola}{2} \,
    \volball_{d-1} ~ \rdep\rspace ^{d-1} 
    \sum_{j=1}^{+\infty} \Big( 1-\frac{|\Grande_i -\Grande_j|^2(t)}{4\rdep\rspace ^2} \Big)_{\hspace{-1mm}+}^{\hspace{-1mm}\frac{d-1}{2}} 
                          \frac{\Grande_i -\Grande_j}{|\Grande_i -\Grande_j|}(t)\, d t 
                          \\
    \phantom{d \Grande_i (t) = }
    + \disp \sum_{j=1}^{+\infty}\big(\Grande_i-\Grande_j\big)(t) d L_{ij}(t) \, , \\
    \disp
    \text{for any } j\in\N^*,\ L_{ij}(0) = 0,\ L_{ij} \equiv L_{ji},\\
    \disp \int_0^t \un_{|\Grande_i (s)-\Grande_j(s)|\neq 2 \,\rgrande} \, d  L_{ij}(s) = 0,\ L_{ii} \equiv 0,
\end{array}
\end{cases}	
\end{equation}
where $(\WGrande_i)_i$ are independent $\R^d$-valued Brownian motions and the local times $L_{ij}$ describe the effects of the elastic collisions between the hard spheres $i$ and $j$ (subject to normal reflection).

Moreover, any $\mugrande\in\Gcal_{\zgrande}(\zpiccola\nrg)$ is reversible for the dynamics.
This proves Theorem~\ref{th:DepletionDynamics}.

\subsection{Low-activity regime: no sphere-percolation}\label{sec:percolation}
%%%%%%%%%%%%%%%%%%%%%%%%%%%%%%%%%%%%%%%%%%%%%%%%%%%%%%%%%%%%%%%%%%%%%%

As a by-product of the chain estimate of Lemma~\ref{prop:LongChainDynamics}, needed to control the dynamics and prove the existence of solutions for~\eqref{eq:SDE2infty}, we get absence of percolation $\mugrande$-a.s.\ for small activities $\zgrande$. Although our argument is not the classical one in the line of~\cite{Liggett_Schonmann_Stacey_1997}, it still relies on the crucial idea that there is only a very small area in which an additional sphere can be put in order to lengthen a sphere chain, hence a very small, geometrically decreasing, probability to have long chains. 

\begin{proposition}\label{prop:No percolation}
We say that two spheres $\grande_i$ and $\grande_j$ of a configuration $\bgrande$ \emph{interact} if 
$|\grande_i - \grande_j| < 2\rdep$, i.e.\ if their interaction $\Vovlap\Big(\frac{|\grande_i - \grande_j|}{2 \rdep}\Big)$ does not vanish.
Let $\mugrande\in\Gcal_{\zgrande}(\zpiccola\nrg)$. Then, for any $\zpiccola>0$ 
and any $\zgrande < \zgrande_{\rm c} := (2^d(\rdep\rspace^d-\rgrande^d)\volball_d)^{-1}$,
there is $\mugrande$-a.s. no infinite cluster of interacting spheres. 
\end{proposition}  

\begin{proof}
If a configuration $\bgrandey$ has an infinite chain of interacting spheres at distance smaller than $2\rdep$ of one another, there is a radius $R(\bgrandey)$ such that $B(0,R)$ intersects this infinite chain for each $R>R(\bgrandey)$. 
Consequently, $\bgrandey$ has a chain of interacting spheres from $B(0,R)$ to $B(0,2R)$
which involves at least $\floor*{\frac{R}{2\rdep}}$ spheres. 
We then only have to prove that
\begin{align*}
 \mugrande(C_\infty)=0,
 \end{align*}
where
\begin{align*}
 C_\infty := \left\{ \text{For $R$ large enough, }~~ 
                     \bgrandey_{B(0,2R)} \in \BadPath_{\text{Chain}}(R,\floor*{\frac{R}{2\rdep}},2\rpiccola) \right\}.
\end{align*}
This holds as soon as 
\begin{align}\label{eq:SummabilityPercolation}
 \sum_R \mugrande\left( \bgrandey_{B(0,2R)} \in \BadPath_{\text{Chain}}(R,\floor*{\frac{R}{2\rdep}},2\rpiccola)  \right) \quad <+\infty .
\end{align}
By Proposition~\ref{prop:tracemesureequ}, $\mugrande$ is the projection on the sphere configuration space $\MGrande$ of the Gibbs measure 
$\mugrande(d\bgrande)\otimes\pi^{\zpiccola}_{\R^d\setminus\BGrande(\bgrande)}(d\bpiccola)\in\Gcal_{\zgrande,\zpiccola}$ on $\M$. 
Since $\BadPath_{\text{Chain}}$ events only involve spheres,
and $X^{\by,2R}(0)=\by_{B(0,2R)}$ by~\eqref{eq:XyR} is $\mu^\by_{2R}$-distributed under $Q^\by_{2R}$, we have
\begin{align*}
 &\mugrande\left( \bgrandey_{B(0,2R)} \in \BadPath_{\text{Chain}}(R,\floor*{\frac{R}{2\rdep}},2\rpiccola) \right)\\
 & = \mu\left( \by_{B(0,2R)} \in \BadPath_{\text{Chain}}(R,\floor*{\frac{R}{2\rdep}},2\rpiccola) \right) \\
 & = \int_\D P\left( X^{\by,2R}(0) \in \BadPath_{\text{Chain}}(R,\floor*{\frac{R}{2\rdep}},2\rpiccola)  \right) \mu(d\by) \\
 & \le \int_\D Q^\by_{2R}\left( \BadPath'_{\text{Chain}}(1,R,\floor*{\frac{R}{2\rdep}},2\rpiccola) \right) \mu(d\by) + \mathbf{d}_{TV}(2R).
\end{align*}
Thanks to~\eqref{eq:TotalVariationUpperBound} and Lemma~\ref{prop:LongChainDynamics},
we obtain
\begin{align*}
    &\mugrande\left( \bgrandey_{B(0,2R)} \in \BadPath_{\text{Chain}}(R,\floor*{\frac{R}{2\rdep}},2\rpiccola) \right) \\
    &\le \left( \zgrande \, ((2\rdep)^d-(2\rgrande)^d) \volball_d \right)^{\floor*{\frac{R}{2\rdep}}}  
      + 2\int_\D \zgrande \lambdaGyR((\BallGyR)^c) +\zpiccola \lambdaPyR((\BallPyR)^c) \, \mu(d\by).
\end{align*}
This ensures that~\eqref{eq:SummabilityPercolation} holds as soon as 
$\disp \zgrande < \frac{1}{((2\rdep)^d-(2\rgrande)^d) \volball_d}$, completing the proof.
\end{proof}

\begin{remark}\label{PercolationAsymtotics}
In a realistic setting, the radius $\rpiccola$ of the particles should be very small.
As expected, the colloidal-percolation critical activity 
\begin{equation}
    \zgrande_{\rm c} 
   = \frac{1}{2^d \volball_d \sum_{k=1}^d \binom{d}{k} \rpiccola^k \rgrande^{d-k}}
   \mathop{\quad\sim\quad}_{\rpiccola\to0} \frac{1}{d 2^d \volball_d \, \rgrande^{d-1} ~ \rpiccola }
\end{equation}
tends to infinity as the particle radius tends to zero.
Indeed, by definition, as $\rpiccola$ gets smaller, the distance at which sphere are considered interacting get smaller, hence it is more difficult for interacting chains to appear. For extremely small $\rpiccola$'s, such chains simply cannot form.
\end{remark}

\begin{remark}
    Although the connection between (absence of) percolation and (absence of) phase transition is not automatic -- in particular, the FKG property does not hold and, since this is not a symmetric mixture, a phase transition cannot be proven via symmetry breaking using percolation as in~\cite{CCK,Giacomin1995Agreement} -- it is still interesting to compare the different regimes.
    
    Let $\phi(y) = \volball_d\rdep\rspace^d - \frac12 \Vovlap(\abs{y}/(2\rdep))\geq 0$. By the classical result of Ruelle~\cite[Theorem 4.2.3]{Ruelle_book}, we know that there is no phase transition if
    \begin{equation*}
        \zgrande<e^{-1}\left((2\rgrande)^d \volball_d +
        \int_{\abs{y}\in(2\rgrande,2\rdep)}\abs{1-e^{-\zpiccola\phi(y)}}dy\right)^{-1}.
    \end{equation*}
    Comparing with Proposition~\ref{prop:No percolation}, we see (apart from the $e^{-1}$ factor) an additional $(2\rgrande)^d$ at the denominator, coming from the hard core. It is however worth noting (see~\cite{JT19}) that this should be significantly improved by considering the \emph{renormalised activity} $\zgrande\rspace' = \zgrande\exp\{-\zpiccola \volball_d \rdep\rspace^d\}$.
\end{remark}

\subsection{High-activity regime: towards an optimal packing}\label{sec:packing}
%%%%%%%%%%%%%%%%%%%%%%%%%%%%%%%%%%%%%%%%%%%%%%%%%%%%%%%%%%%%%%%%%%%%%%%%%%%%%%%%%%

We are now interested in the high-density behaviour of the equilibrium measures in $\Gcal_{\zgrande}(\zpiccola\nrg)$.

In the case of a finite number of hard spheres, we were able to prove in~\cite{FKRZ24} that
the reversible probability measure of the corresponding finite-dimensional gradient system for $n$ hard spheres concentrates around admissible configurations which minimise the energy.
Moreover, admissible $n$-sphere configurations who realise the minimal energy maximise the contact number. 

In the infinite setting, the behaviour is more complicated, as the number of spheres in any volume is random. There is now a competition between the minimisation of the \emph{energy} (achieved either by having less spheres or having the more densely packed) and that of the \emph{entropy} (achieved by being as close as possible to the completely random Poisson point process); this is the \emph{Gibbs variational principle}~\cite{georgii_1994,georgii_1995}. 
In particular then, since the energy is minimised not only by closely packing a given number of spheres, but also by taking as few of them as possible, only in the limit $\zgrande\to\infty$, the entropic cost of putting only a small number of spheres in the system is too high, making the minimising configurations only those that achieve the closest packing. This heuristics can be formalised as follows:
\begin{proposition}
    Let $\mugrande\in\Gcal_{\zgrande}(\zpiccola\nrg)$, and denote is intensity by $\rho_\zgrande(\zpiccola) = \lim_{n\to\infty}\frac{1}{\abs{\Lambda_n}}\int \abs{\bgrande_{\Lambda_n}}d \mugrande(\bgrande)$, $\Lambda_n=[-n,n)^d$. Then, for any $\zpiccola>0$, as $\zgrande\to\infty$, $\mugrande$ attains the closest-packing density, that is:
    \begin{equation*}
        \lim_{\zgrande\to\infty} \rho_\zgrande(\zpiccola) = \lim_{n\to\infty}\frac{1}{\abs{\Lambda_n}} N_n =: \rho^*,
    \end{equation*}
    with $N_n$ the maximal number of mutually disjoint open balls of radius $\rgrande$ included in $\Lambda_n$.
\end{proposition}
The proof is a straightforward application of~\cite[Proposition 2]{Mase_2001}, which states the above result for any regular potential with hard core, as $\phi_2$ is (e.g.~\cite{Ruelle_book}), using the Gibbs variational principle as main tool. Note that the effects of the ``finite'' parts of the interaction vanish when taking the high-activity limit, and only the infinite hard-core interaction plays a role; in particular, the limiting density does not depend on the inverse temperature parameter $\zpiccola$, and is simply given by the packing density, i.e.\ the highest possible one under the hard-core constraint.
\begin{remark}
    In view of such a result, it is natural to ask whether one may simulate the gradient dynamics~\eqref{eq:Sdep} to estimate the celebrated contact number between $n$ hard spheres in any dimension. Indeed, we have made such a simulation (see~\url{https://lab.wias-berlin.de/zass/dynamics-of-spheres}) for our previous work~\cite{FKRZ24}; obtaining estimates in high dimensions is then of course worth pursuing, but requires long computation times.
\end{remark}

%%%%%%%%%%%%%%%%%%%%%%%%%%%%%%%%%%%%%%%%%%%%%%%%%%%%%%%%%%%%%%%%%%%
\appendix
\section{Proof of the remaining statements}
%%%%%%%%%%%%%%%%%%%%%%%%%%%%%%%%%%%%%%%%%%%%%%%%%%%%%%%%%%%%%%%%%%%

\subsection{Existence of the penalisation functions}\label{app:psigrandepsipiccola} %%%%%%%%%%%%%%%%%%%%%%%%%%%%%

We prove here the existence of the sphere penalisation function $\psiGyR$ 
and the particle penalisation function $\psiPyR$ needed in Section \ref{sebsect:penalisationFunctions}.
Let $\psi$ be a non-increasing $\C^\infty$ function with $\psi(t)=1$ if $t\le0$ and $\psi(t)=0$ if $t\ge1$, and $\phi$ be a non-decreasing $\C^\infty$ function with $\phi(t)=0$ if $t\le0$ and $\phi(t)=t+\const$ if $t\ge1$. Then the two functions
\begin{align*}
  \psiGyR(x) := 
  & 2\log(R) + \phi\big(R^{d+1}(|x|-R)\big) 
    +\sum_{\substack{\grandey\in\by \\ R\le|\grandey|\le R+2\rgrande}} 
     \psi\left( \frac{|\grandey-x|^2}{4\rgrande^2} \right)+\sum_{\substack{\piccolay\in\by \\ R\le|\piccolay|\le R+\rdep}} 
     \psi\left( \frac{|\piccolay-x|^2}{\rdep\rspace ^2} \right) \\
    & \qquad +\log\Big( \sharp\{ \grandey\in\by: R\le|\grandey|\le R+2\rgrande \} \Big) + \log\Big( \sharp\{ \piccolay\in\by: R\le|\piccolay|\le R+\rdep \} \Big), \\
 \psiPyR(x) := 
  & 2\log(R) + \phi\big(R^{d+1}(|x|-R)\big) 
    +\sum_{\substack{\grandey\in\by \\ R\le|\grandey|\le R+\rdep}} 
     \psi\left( \frac{|\grandey-x|^2}{\rdep\rspace ^2} \right)\\
  & \qquad +\log\Big( \sharp\{ \grandey\in\by: R\le|\grandey|\le R+\rdep \} \Big),
\end{align*} 
with the convention that $\log(\sharp\emptyset)+\sum_\emptyset \cdots =0$, are non-negative functions of class $\mathcal{C}^2$ with bounded derivatives. Moreover, $\psiGyR$ is constant on $\BallGyR$, and $\psiPyR$ is constant on $\BallPyR$. By construction then,
\begin{equation*}
  \sum_{R=1}^\infty \int_{(\BallGyR)^c} e^{-\psiGyR(x)} \, dx <+\infty
  \quad\text{ and }\quad 
  \sum_{R=1}^\infty \int_{(\BallPyR)^c} e^{-\psiPyR(x)} \, dx <+\infty.
\end{equation*}

\subsection{Proof of the ball separation and nested inclusion Lemma \ref{LemmaNestedInclusion}}\label{app:NestedInclusionProof} %%%%%%%%%%%%%%%%%

We consider a path $X:[0,1]\to\D$ that belongs neither to $\BadPath'_{\text{Chain}}(\delta,\alpha,\kappa,\eps)$ nor to $\BadPath'_{\text{Fast}}(\alpha,\delta,\eps/4)$, and study it during the time interval $[a\delta,(a+1)\delta]$, on balls of radii
\begin{equation*}
\begin{split}
    0&< \rho=:\rho_{1/\delta} < \cdots 
    < \rho_{a+1} < \rho_a:=\rho_{a+1}+2\kappa(2\rgrande+\eps) 
    < \cdots 
    < \rho_0:=\rho+\frac{2\kappa}{\delta}(2\rgrande+\eps) 
    \le \alpha, 
\end{split}
\end{equation*}
focusing first on the spheres, then on the particles.

By definition of the index set $\Igrande := \Igrande(X(a\delta),\rho_a,\eps)$, the spheres of $X(a\delta)$ whose indices are not in $\Igrande$ are neither in the ball $B(0,\rho_a)$ nor in an $\eps$-chain connected to it, i.e.\ not at distance $(2\rgrande+\eps)$ of any $\eps$-chain connected to $B(0,\rho_a)$, that is:
\begin{equation*}
    \forall j\notin\Igrande(X(a\delta),\rho_a,\eps), \quad 
   |\grandeX_j(a\delta)| > \rho_a,
\end{equation*}
\begin{equation*}
   \forall j\notin\Igrande(X(a\delta),\rho_a,\eps),\  
   \forall i\in\Igrande(X(a\delta),\rho_a,\eps),\quad 
   |\grandeX_j(a\delta)-\grandeX_i(a\delta)| >2\rgrande+\eps.
\end{equation*}
$X\notin\BadPath'_{\text{Fast}}(\alpha,\delta,\eps/4)$ implies that no sphere in $B(0,\alpha)$ moves by more than $\eps/4$ 
during the time interval $[a\delta,(a+1)\delta]$, hence the separation property:
\begin{equation*}
\begin{split}
    &\forall j\notin\Igrande(X(a\delta),\rho_a,\eps),\ 
    \forall i\in\Igrande(X(a\delta),\rho_a,\eps),\ 
    \forall t\in[a\delta,(a+1)\delta], \\ 
    &\qquad |\grandeX_j(t)-\grandeX_i(t)| >2\rgrande+\frac{\eps}{2}.
\end{split}  
\end{equation*}
Since $X\notin \BadPath'_{\text{Chain}}(\delta,\alpha,\kappa,\eps)$, the starting configuration $X(a\delta)$ does not have any $\eps$-chain of spheres that intersects $B(0,\alpha)$ and involves $(\kappa+1)$ spheres or more, i.e.\ is longer than $\kappa(2\rgrande+\eps)$. 
Consequently, all spheres labelled in $\Igrande(X(a\delta),\rho_a,\eps)$
start at a distance from the origin of at most:
\begin{equation}\label{RadiusJSpheres}
   \forall i\in\Igrande(X(a\delta),\rho_a,\eps),\quad 
   |\grandeX_i(a\delta)| 
   \le \underbrace{ ~~~\rho_a~~~ }_{\text{radius defining $\Igrande$}} 
       + \underbrace{ \kappa(2\rgrande+\eps) }_{\text{longest $\eps$-chain}}
\end{equation}   
Since they move at most of $\eps/4$ during the time interval $[a\delta,(a+1)\delta]$, we get the localisation property:
\begin{equation*}
    \forall i\in\Igrande(X(a\delta),\rho_a,\eps),\ 
   \forall t\in[a\delta,(a+1)\delta], \quad 
   |\grandeX_i(t)| \le \rho_a+ \kappa(2\rgrande+\eps) +\frac{\eps}{4}.  
\end{equation*}
Finally, for each $j\notin\Igrande(X(a\delta),\rho_a,\eps)$,
we already have $|\grandeX_j(a\delta)| > \rho_a$,
thus $|\grandeX_j((a+1)\delta)| > \rho_a-\frac{\eps}{4}$:
\begin{equation*}
\begin{split}
    j\notin\Igrande(X(a\delta),\rho_a,\eps) 
    &\implies |\grandeX_j((a+1)\delta)| > \rho_a-\frac{\eps}{4} 
                                      > \rho_{a+1} + \kappa(2\rgrande+\eps) \\
   &\implies j\notin\Igrande(X((a+1)\delta),\rho_{a+1},\eps), 
\end{split}   
\end{equation*}
thanks to~\eqref{RadiusJSpheres}, used for $a+1$ instead of $a$.
This proves the first inclusion:
\begin{equation*}
    \Igrande(X((a+1)\delta),\rho_{a+1},\eps) \subset \Igrande(X(a\delta),\rho_a,\eps).  
\end{equation*}
The proof of the analogue properties for the particles is similar.
The particles of $X(a\delta)$ whose indices do not belong to $\Ipiccola:= \Ipiccola(X(a\delta),\rho_a,\eps) $ are neither in $B(0,\rho_a)$ nor in the $\eps$-neighbourhood of an $\eps$-chain connected to it, i.e.\ not at distance $\rdep+\eps$ of any $\eps$-chain connected to $B(0,\rho_a)$:
\begin{equation*}
    \forall k\notin\Ipiccola(X(a\delta),\rho_a,\eps), \quad 
    |\piccolaX_k(a\delta)| > \rho_a,
\end{equation*}
\begin{equation*}
    \forall k\notin\Ipiccola(X(a\delta),\rho_a,\eps),\ 
    \forall i\in\Igrande(X(a\delta),\rho_a,\eps), \quad 
    |\piccolaX_k(a\delta)-\grandeX_i(a\delta)| >\rdep+\eps .  
\end{equation*}
Again, no sphere and no particle in $B(0,\alpha)$ moves by more than $\eps/4$ 
during the time interval, hence 
\begin{equation}\label{RadiusJParticles}
  \forall k\notin\Ipiccola(X(a\delta),\rho_a,\eps),\ 
   \forall t\in[a\delta,(a+1)\delta], \quad |\piccolaX_k(t)| >\rho_a-\frac{\eps}{4},
\end{equation}  
\begin{equation*}
\begin{split}
    &\forall k\notin\Ipiccola(X(a\delta),\rho_a,\eps),\ 
    \forall i\in\Igrande(X(a\delta),\rho_a,\eps),\ 
    \forall t\in[a\delta,(a+1)\delta],\quad |\piccolaX_k(t)-\grandeX_i(t)| >\rdep+\frac{\eps}{2}.
\end{split}
\end{equation*}
Particles in $\Ipiccola$ stay near the origin as a consequence of~\eqref{RadiusJSpheres},
\begin{equation*}
    \forall k\in\Ipiccola(X(a\delta),\rho_a,\eps), \quad 
   |\piccolaX_k(a\delta)| \le \underbrace{ \rho_a + \kappa(2\rgrande+\eps)  }_{\text{farthest position of spheres in $\Igrande$}} + \rdep+\eps,  
\end{equation*}
hence 
\begin{equation*}
    \forall k\in\Ipiccola(X(a\delta),\rho_a,\eps),\ 
   \forall t\in[a\delta,(a+1)\delta], \ 
   |\piccolaX_k(t)| \le \rho_a + \kappa(2\rgrande+\eps) + \rdep+\frac{5\eps}{4}. 
\end{equation*}
Inequality~\eqref{RadiusJParticles} for $t=(a+1)\delta$ yields
\begin{equation*}
\begin{split}
    &k\notin\Ipiccola(X(a\delta),\rho_a,\eps) \\
    &\implies 
    |\piccolaX_k((a+1)\delta)| > \rho_a-\frac{\eps}{4} 
    \ge \rho_{a+1} +\kappa(2\rgrande+\eps) +\rdep+\eps \\
    &\implies
    k\notin\Ipiccola(X((a+1)\delta),\rho_{a+1},\eps),
\end{split}
\end{equation*}
thanks to the consequence of~\eqref{RadiusJSpheres}, applied to $(a+1)$ instead of $a$.
This proves the second inclusion:
\begin{equation*}
    \Ipiccola(X((a+1)\delta),\rho_{a+1},\eps) 
   \subset \Ipiccola(X(a\delta),\rho_a,\eps),   
\end{equation*}
thus concluding the proof of the lemma.
\qed

\subsection{Proof of Lemma~\ref{prop:LongChainDynamics}: long sphere chains are improbable}\label{sec:proofLongChain} %%%%%%%%%%%%%%%%%%%%%%%%%%%%% 

We prove here Lemma~\ref{prop:LongChainDynamics}, that is, we find an upper bound for the probability of long sphere chains to form, at some time between $0$ and $1$, for the mixed penalised process starting from its reversible distribution.

The definition of $\BadPath'_{\text{Chain}}(\delta,\alpha,\kappa,\eps)$ in Definition \ref{def:defChain} and the fact that $X(t)$ is $\mu^\by_R$-distributed for each $t$ under $Q^\by_R$ imply, for any time step $\delta>0$\,:
\begin{align*}
   Q^\by_R\Big( \BadPath'_{\text{Chain}}(\delta,\alpha,\kappa,\eps) \Big) 
   & = \bigcup_{k=0}^{\lfloor 1/\delta \rfloor} Q^\by_R\Big( \bX(k\delta)\in\BadPath'_{\text{Chain}}(\alpha,\kappa,\eps) \Big) 
    \le \frac{1}{\delta} ~ \mu^\by_R\left( \BadPath_{\text{Chain}}(\alpha,\kappa,\eps) \right).
\end{align*}
To complete the proof of Lemma~\ref{prop:LongChainDynamics}, we only have to prove the following static estimate:
\begin{lemma}\label{lm:LongChainEstimate} For each sphere activity $\zgrande>0$, any chain length $\kappa\in\N^*$, any radius $\alpha>0$, and any leeway $\eps\in(0,2\rgrande)$\,,
\begin{align*}
\forall \by\in\D,\, \forall R>0, \quad
 \mu^\by_R\left( \BadPath_{\text{Chain}}(\alpha,\kappa,\eps) \right)  
 & \le \left( \zgrande \, ((2\rgrande+\eps)^d-(2\rgrande)^d)\volball_d \right)^\kappa .
\end{align*}
\end{lemma}

\begin{proof}
The proof relies on an induction inequality for chains under the measure $\nu^\by_{R,\ngrande,\npiccola}$.

We have to carefully manage the difference between labelled and unlabelled sphere configurations. Fix first the sphere indices in the norm, in order to avoid double counting. This is possible because the measure of the set of configurations where several spheres have the same norm is zero. 
The exchangeability of the measure $\otimes_{i=1}^\ngrande \lambdaGyR(d\grande_i)$ ensures that:
\begin{align*}
& \nu^\by_{R,\ngrande,\npiccola}\left( \BadPath_{\text{Chain}}(\alpha,\kappa,\eps) \right)   \\
& = n! \int_{\R^{\ngrande d}} \int_{\R^{\npiccola d}}
       \un_{\{|\grande_1|<|\grande_2|<\cdots<|\grande_n|\}} \, 
       \un_{\{ ( \grande_1, \ldots, \grande_n )\in\BadPath_{\text{Chain}}(\alpha,\kappa,\eps)\}} \, 
       \un_{\D}(\bgrande\bpiccola)       
\otimes_{k=1}^\npiccola \lambdaPyR(d\piccola_k) \otimes_{i=1}^\ngrande \lambdaGyR(d\grande_i).     
\end{align*}
For any finite configuration $\bgrande = (\grande_1,\ldots,\grande_\ngrande) \in \D$, 
the existence of at least one $\eps$-chain of length $\kappa$ starting in $|B(0,\alpha)|$ is equivalent to the existence of a permutation $\tau$ on $\{1,\dots,\ngrande\}$ such that
$|\grande_{\tau(1)}|\le\alpha$ and
$|\grande_{\tau(1)}-\grande_{\tau(2)}|<2\rgrande+\eps$, 
$|\grande_{\tau(2)}-\grande_{\tau(3)}|<2\rgrande+\eps$, \ldots,
$|\grande_{\tau(\kappa)}-\grande_{\tau(\kappa+1)}|<2\rgrande+\eps$.
Using the inverse permutation, again denoted by $\tau$, we can rewrite the expression:
\begin{align}
& \nu^\by_{R,\ngrande,\npiccola}\left( \BadPath_{\text{Chain}}(\alpha,\kappa,\eps) \right)   \\
& = n! \int_{\R^{\ngrande d}} \int_{\R^{\npiccola d}}
       \un_{\{|\grande_1|<|\grande_2|<\cdots<|\grande_n|\}}\       
       \un_{\{ \exists \tau : 
             |\grande_{\tau(1)}|\le\alpha,~
             |\grande_{\tau(1)}-\grande_{\tau(2)}|<2\rgrande+\eps,\ldots,~
             |\grande_{\tau(\kappa)}-\grande_{\tau(\kappa+1)}|<2\rgrande+\eps \}} \, \un_{\D}(\bgrande\bpiccola)  \\
&\phantom{= n! \int_{\R^{\ngrande d}} \int_{\R^{\npiccola d}}\un_{|\grande_1|<|\grande_2|<\cdots<|\grande_n|} \, \un}
        \qquad\qquad\qquad\qquad\qquad\qquad\qquad\,  \otimes_{k=1}^\npiccola \lambdaPyR(d\piccola_k) \otimes_{i=1}^\ngrande \lambdaGyR(d\grande_i) \\
& = n! \int_{\R^{\ngrande d}} \int_{\R^{\npiccola d}}
       \un_{\{\exists \tau :
            |\grande_{\tau(1)}|<|\grande_{\tau(2)}|<\cdots<|\grande_{\tau(n)}|\}} \
       \un_{\{ |\grande_1|\le\alpha,~
             |\grande_1-\grande_2|<2\rgrande+\eps,\ldots,~
             |\grande_{\kappa}-\grande_{\kappa+1}|<2\rgrande+\eps \}} \,  \un_{\D}(\bgrande\bpiccola) \\
&\phantom{= n! \int_{\R^{\ngrande d}} \int_{\R^{\npiccola d}}\un_{|\grande_1|<|\grande_2|<\cdots<|\grande_n|} \, \un}
        \qquad\qquad\qquad\qquad\qquad\qquad\qquad\,  \otimes_{k=1}^\npiccola \lambdaPyR(d\piccola_k) \otimes_{i=1}^\ngrande \lambdaGyR(d\grande_i).
\end{align}
Integrating separately with respect to $\grande_{\kappa+1}$, we find:
\begin{align}
&\nu^\by_{R,\ngrande,\npiccola}\left( \BadPath_{\text{Chain}}(\alpha,\kappa,\eps) \right)\\
& = n! \int_{\R^{(\ngrande-1) d}} \int_{\R^{\npiccola d}} \int_{\R^{d}}
       \un_{ \{|\grande_1|\le\alpha,~
             |\grande_1-\grande_2|<2\rgrande+\eps,\ldots,~
             |\grande_{\kappa}-\grande_{\kappa+1}|<2\rgrande+\eps\} } \, \un_{\D}(\bgrande\bpiccola)  \\
&\phantom{= n! \int_{\R^{\ngrande d}} \int_{\R^{\npiccola d}}\un}
        \qquad \qquad
       \lambdaGyR(d\grande_{\kappa+1}) \otimes_{k=1}^\npiccola \lambdaPyR(d\piccola_k) \otimes_{i=1,i\neq\kappa+1}^\ngrande \lambdaGyR(d\grande_i),\label{eq:nunmChain}
\end{align}
since $\un_{\{\exists \tau : |\grande_{\tau(1)}|<|\grande_{\tau(2)}|<\cdots<|\grande_{\tau(n)}|\}}$ is equal to $1$ for almost every sphere configuration. 
Thanks to the inequality
\begin{equation*}
    \un_{\D} (\grande_1,\dots,\grande_{\ngrande}, \bpiccola)
   \le \un_{\D}( \grande_1,\dots,\grande_{\kappa},\grande_{\kappa+2},\dots,\grande_{\ngrande}, \bpiccola) \, 
   \un_{\{|\grande_{\kappa}-\grande_{\kappa+1}|\ge2\rgrande\}},
\end{equation*}
the integral on $\grande_{\kappa+1}$ can be isolated as
\begin{equation*}
\begin{split}
    &\int_{\R^{d}} \un_{\{ 2\rgrande \le |\grande_{\kappa}-\grande_{\kappa+1}| <2\rgrande+\eps\} } \, 
   \lambdaGyR(d\grande_{\kappa+1}) 
   \le |B(\grande_{\kappa},2\rgrande+\eps) \setminus B(\grande_{\kappa},2\rgrande)|
       = ((2\rgrande+\eps)^d-(2\rgrande)^d)|B(0,1)| .
\end{split}
\end{equation*}
We finally get
\begin{align*}
    & \nu^\by_{R,\ngrande,\npiccola}\left( \BadPath_{\text{Chain}}(\alpha,\kappa,\eps) \right)   \\
    & \le n! ((2\rgrande+\eps)^d-(2\rgrande)^d)|B(0,1)|\\
&\phantom{\le }
    \int_{\R^{(\ngrande-1) d}} \int_{\R^{\npiccola d}} 
    \un_{\{ |\grande_1|\le\alpha,\,
             |\grande_1-\grande_2|<2\rgrande+\eps,\ldots,\,
             |\grande_{\kappa-1}-\grande_{\kappa}|<2\rgrande+\eps\} } \, \un_{\D}(\bgrande\bpiccola)  \otimes_{k=1}^\npiccola \lambdaPyR(d\piccola_k) \otimes_{i=1,i\neq\kappa+1}^\ngrande \lambdaGyR(d\grande_i).
\end{align*}
From~\eqref{eq:nunmChain}, the above integral is equal to 
\begin{equation*}
    (\ngrande-1)!~ \nu^\by_{R,\ngrande-1,\npiccola}\left( \BadPath_{\text{Chain}}(\alpha,\kappa-1,\eps) \right),
\end{equation*}
so that we find
\begin{align*}
    &\nu^\by_{R,\ngrande,\npiccola}\left( \BadPath_{\text{Chain}}(\alpha,\kappa,\eps) \right) 
    \le n ~ ((2\rgrande+\eps)^d-(2\rgrande)^d)|B(0,1)| ~ \nu^\by_{R,\ngrande-1,\npiccola}\left( \BadPath_{\text{Chain}}(\alpha,\kappa-1,\eps) \right).
\end{align*}
Let us now transfer this chain induction inequality to the mixed probability measure $\mu^\by_R$.
According to~\eqref{eq:muRy}, and since configurations with less than $\kappa+1$ spheres do not have chains of length $\kappa$:
\begin{align*}
& \mu^\by_R\left( \BadPath_{\text{Chain}}(\alpha,\kappa,\eps) \right)   \\
& = \frac{1}{\Zmixing^\by_R} 
    \sum_{\npiccola=0}^{+\infty} \frac{\zpiccola^\npiccola}{\npiccola!} \, 
    \sum_{\ngrande=\kappa+1}^{+\infty} \frac{\zgrande^\ngrande}{\ngrande!} \,                
    \nu^\by_{R,\ngrande,\npiccola} \left( \BadPath_{\text{Chain}}(\alpha,\kappa,\eps) \right) \\
& \le ((2\rgrande+\eps)^d-(2\rgrande)^d)|B(0,1)| \, \frac{1}{\Zmixing^\by_R}
      \sum_{\npiccola=0}^{+\infty} \frac{\zpiccola^\npiccola}{\npiccola!} \,
      \sum_{\ngrande=\kappa+1}^{+\infty} \frac{\zgrande^\ngrande}{\ngrande!} \, ~ 
      \ngrande \, \nu^\by_{R,\ngrande-1,\npiccola}\left( \BadPath_{\text{Chain}}(\alpha,\kappa-1,\eps) \right) \\
& \le ((2\rgrande+\eps)^d-(2\rgrande)^d)\volball_d \, 
      \frac{\zgrande}{\Zmixing^\by_R} \,
      \sum_{\npiccola=0}^{+\infty} \frac{\zpiccola^\npiccola}{\npiccola!} \,
      \sum_{\ngrande=\kappa}^{+\infty} \frac{\zgrande^\ngrande}{\ngrande!} \,                
      \nu^\by_{R,\ngrande,\npiccola}\left( \BadPath_{\text{Chain}}(\alpha,\kappa-1,\eps) \right) \\
& \le \zgrande \, ((2\rgrande+\eps)^d-(2\rgrande)^d)\volball_d \,
      \mu^\by_R\left( \BadPath_{\text{Chain}}(\alpha,\kappa-1,\eps) \right) .      
\end{align*}
A straightforward induction argument completes the proof of Lemma \ref{lm:LongChainEstimate}.
\end{proof}

\subsection{Proof of Lemma~\ref{prop:OmegaFast}: fast moving spheres or particles are improbable}\label{sec:proofFastMotion} %%%%%%%%%%%%%%%%%%%%%%%%%%%%%%%%

The oscillation estimates of the penalised processes stated in Lemma \ref{prop:OmegaFast} rely on time reversal techniques. 

Let us first write the reversible penalised process as a sum of forward and backward Brownian motions.
According to~\eqref{eq:SDEnSpheres}, the following processes are rescaled Brownian motions for $1\le i,j \le \ngrande$ and $1\le k \le \npiccola$:
%\small
\begin{equation*}
\begin{cases} \disp
 \WGrande_i(t) 
 = \Grande_i^{\by,R,\ngrande,\npiccola}(t) -\Grande_i^{\by,R,\ngrande,\npiccola}(0) 
   +\frac12\int_0^t \nabla\psiGyR \big(\Grande_i^{\by,R,\ngrande,\npiccola}(s)\big) \,ds \\ \disp\qquad\qquad
   -\sum_{j=1}^\ngrande \int_0^t \big(\Grande_i^{\by,R,\ngrande,\npiccola}-\Grande_j^{\by,R,\ngrande,\npiccola}\big)(s) \,dL_{ij}^{\by,R,\ngrande,\npiccola}(s) \\
    \disp\qquad\qquad
   -\sum_{k=1}^\npiccola \int_0^t \big(\Grande_i^{\by,R,\ngrande,\npiccola}-\Piccola_k^{\by,R,\ngrande,\npiccola}\big)(s) \,d\ell_{ik}^{\by,R,\ngrande,\npiccola}(s) \\ \disp 
 \sigmaP\, \WPiccola_k(t) 
 = \Piccola_k^{\by,R,\ngrande,\npiccola}(t) -\Piccola_k^{\by,R,\ngrande,\npiccola}(0) 
   +\frac{\sigmaP^2}{2} \int_0^t \nabla\psiPyR\big(\Piccola_k(s)\big) \,ds\\
    \disp\qquad\qquad
   -\sigmaP^2 \sum_{i=1}^\ngrande \int_0^t \big(\Piccola_k^{\by,R,\ngrande,\npiccola}-\Grande_i^{\by,R,\ngrande,\npiccola}\big)(s) \,d\ell_{ik}^{\by,R,\ngrande,\npiccola}(s).
\end{cases}
\end{equation*}
\normalsize
The forward process 
\begin{equation*}
    \big( \Grande_i^{\by,R,\ngrande,\npiccola}(t), \Piccola_k^{\by,R,\ngrande,\npiccola}(t),  L_{ij}^{\by,R,\ngrande,\npiccola}(t), \ell_{ik}^{\by,R,\ngrande,\npiccola}(t) \big)_{t\in [0,1], 1\le i,j \le \ngrande,\, 1\le k \le \npiccola} 
\end{equation*}
has the same distribution as the backward process
\begin{equation*}
   \big( \Grande_i^{\by,R,\ngrande,\npiccola}(1-t), \Piccola_k^{\by,R,\ngrande,\npiccola}(1-t),  L_{ij}^{\by,R,\ngrande,\npiccola}(1-t), \ell_{ik}^{\by,R,\ngrande,\npiccola}(1-t) \big)_{t\in [0,1], 1\le i,j \le \ngrande,\, 1\le k \le \npiccola}.
\end{equation*}
Hence, the following images of the backward process also are rescaled Brownian motions:
\begin{equation*}
\begin{cases} 
\disp
 \underleftarrow{\WGrande_i^{\by,R,\ngrande,\npiccola}}(t) 
 = \Grande_i^{\by,R,\ngrande,\npiccola}(1-t) -\Grande_i^{\by,R,\ngrande,\npiccola}(1) 
   \\
   \disp\qquad\qquad\qquad +\frac12\int_0^t \nabla\psiGyR \big(\Grande_i^{\by,R,\ngrande,\npiccola}(1-s)\big) \,ds \\ 
\disp\qquad\qquad\qquad
   -\sum_{j=1}^\ngrande \int_0^t \big(\Grande_i^{\by,R,\ngrande,\npiccola}-\Grande_j^{\by,R,\ngrande,\npiccola}\big)(1-s) \,dL_{ij}^{\by,R,\ngrande,\npiccola}(1-s) \\
\disp\qquad\qquad\qquad
   -\sum_{k=1}^\npiccola \int_0^t \big(\Grande_i^{\by,R,\ngrande,\npiccola}-\Piccola_k^{\by,R,\ngrande,\npiccola}\big)(1-s) \,d\ell_{ik}^{\by,R,\ngrande,\npiccola}(1-s) \\ \disp 
 \sigmaP\, \underleftarrow{\WPiccola_k^{\by,R,\ngrande,\npiccola}}(t) 
 = \Piccola_k^{\by,R,\ngrande,\npiccola}(1-t) -\Piccola_k^{\by,R,\ngrande,\npiccola}(1) \\
 \disp\qquad\qquad\qquad
   +\frac{\sigmaP^2}{2} \int_0^t \nabla\psiPyR\big(\Piccola_k(1-s)\big) \,ds\\
 \disp\qquad\qquad\qquad
   -\sigmaP^2 \sum_{i=1}^\ngrande \int_0^t \big(\Piccola_k^{\by,R,\ngrande,\npiccola}-\Grande_i^{\by,R,\ngrande,\npiccola}\big)(1-s) \,d\ell_{ik}^{\by,R,\ngrande,\npiccola}(1-s). 
\end{cases}
\end{equation*}
\normalsize
As in~\cite{FKRZ24}, the following relations hold for every $t\in[0,1]$\,:
\begin{equation}\label{eq:SumForwardBackward} 
\begin{split}
 \Grande_i^{\by,R,\ngrande,\npiccola}(t)
  = \Grande_i^{\by,R,\ngrande,\npiccola}(0) 
    +\frac12 \left( \WGrande_i(t) +\underleftarrow{\WGrande_i^{\by,R,\ngrande,\npiccola}}(1-t) -\underleftarrow{\WGrande_i^{\by,R,\ngrande,\npiccola}}(1) \right), \\ \disp 
 \Piccola_k^{\by,R,\ngrande,\npiccola}(t)
 = \Piccola_k^{\by,R,\ngrande,\npiccola}(0)
   +\frac{\sigmaP}{2} \left( \WPiccola_k(t) +\underleftarrow{\WPiccola_k^{\by,R,\ngrande,\npiccola}}(1-t) -\underleftarrow{\WPiccola_k^{\by,R,\ngrande,\npiccola}}(1) \right). 
\end{split}
\end{equation}
Let us now compute the probability that, 
among the $\ngrande$ spheres and $\npiccola$ particles in the reversible solution 
$\big( \Grande_i^{\by,R,\ngrande,\npiccola}, \Piccola_k^{\by,R,\ngrande,\npiccola} \big)_{1\le i,j \le \ngrande,\, 1\le k \le \npiccola}$ 
of the finite-dimensional SDE with initial condition distributed according to $\nu^\by_{R,\ngrande,\npiccola}$, 
at least one sphere or particle that enters the $\alpha$-neighbourhood of the origin oscillates too much (see Definition~\ref{def:defFast}):
\begin{equation*}
    Q^\by_{R,\ngrande,\npiccola}\left( \BadPath'_{\text{Fast}}(\alpha,\delta,\eps) \right) 
   =\int P\left( \bX^{\by,R,\ngrande,\npiccola}(\bx,\cdot) \in\BadPath'_{\text{Fast}}(\alpha,\delta,\eps) \right) \, \nu^\by_{R,\ngrande,\npiccola}(d\bx).  
\end{equation*}
The $\ngrande$ spheres (resp. $\npiccola$ particles) are exchangeable, thus
\begin{equation}\label{eq:exchangeability}
\begin{split}
    & Q^\by_{R,\ngrande,\npiccola}\left( \BadPath'_{\text{Fast}}(\alpha,\delta,\eps) \right)\\
    & \le \ngrande\, Q^\by_{R,\ngrande,\npiccola}\left( \Grande_1(\cdot)\in\BadPath_{\text{Fast}}(\alpha,\delta,\eps) \right)
    + \npiccola\, Q^\by_{R,\ngrande,\npiccola}\left( \Piccola_1(\cdot)\in\BadPath_{\text{Fast}}(\alpha,\delta,\eps) \right).
\end{split}
\end{equation}
We evaluate separately first the sphere term and then the particle term. 

According to Definition \ref{def:defFast}, for $\alpha_0=0$ and $\alpha_j=\alpha+\eps/\delta+j/\delta$ if $j\in\N^*$, the event $\Grande_1(\cdot)\in\BadPath_{\text{Fast}}(\alpha,\delta,\eps)$ is equivalent to:
\begin{align*}
    & \exists j,j'\in\N \text{ s.t. } \alpha_j \le |\Grande_1(0)| < \alpha_{j+1} \text{ and } 
                                             \alpha_{j'} \le |\Grande_1(1)| < \alpha_{j'+1} \text{ and } \\
    & \min_{0\le s \le1}|\Grande_1(s)| \leq \alpha \text{~~and } \sup_{\substack{|t-s|<\delta \\ 0\le s,t \le1}} |\Grande_1(t)-\Grande_1(s)| > \eps .
\end{align*}
For a path to go from the outside of the ball with radius $\alpha_j$ to the inside of the ball with radius $\alpha$ in a time duration of less than $1$, the distance covered during at least one of the $1/\delta$ time intervals of length $\delta$ has to be larger than 
$\delta(\alpha_j-\alpha) = \eps+j$. 
Hence $\Grande_1(\cdot)\in\BadPath_{\text{Fast}}(\alpha,\delta,\eps)$ implies:
\begin{align*}
    & \exists j,j'\in\N \text{ s.t. } \alpha_j \le |\Grande_1(0)| < \alpha_{j+1} \text{ and } 
                                      \alpha_{j'} \le |\Grande_1(1)| < \alpha_{j'+1} \text{ and } \\
    & \min_{0\le s \le1}|\Grande_1(s)| \leq \alpha \text{ and } 
      \sup_{\substack{|t-s|<\delta \\ 0\le s,t \le1}} |\Grande_1(t)-\Grande_1(s)|  
      > \eps+\max(j,j').
\end{align*}
Using expression~\eqref{eq:SumForwardBackward}, we get:
\begin{align*}
  & \sup_{\substack{|t-s|<\delta \\ 0\le s,t \le1}} |\Grande_1^{\by,R,\ngrande,  \npiccola}(t)-\Grande_1^{\by,R,\ngrande,\npiccola}(s)| > \eta \\
  & \implies  \sup_{\substack{|t-s|<\delta \\ 0\le s,t \le1}} 
              |\WGrande_1^{\by,R,\ngrande,\npiccola}(t)-\WGrande_1^{\by,R,\ngrande,\npiccola}(s)| > \eta  
            \ \  \text{ or }
\ \ \sup_{\substack{|t-s|<\delta \\ 0\le s,t \le1}} 
                          |\underleftarrow{\WGrande_1^{\by,R,\ngrande,\npiccola}}(t)-\underleftarrow{\WGrande_1^{\by,R,\ngrande,\npiccola}}(s)| 
                          > \eta .
\end{align*}
Thus, thanks to the reversibility of $Q^\by_{R,\ngrande,\npiccola}$, we have:
\begin{align*}
   & Q^\by_{R,\ngrande,\npiccola}\left( \Grande_1^{\by,R,\ngrande,\npiccola}(\cdot)\in\BadPath_{\text{Fast}}(\alpha,\delta,\eps) \right)   \\
   & \le 2 \sum_{j\in\N} 
           Q^\by_{R,\ngrande,\npiccola}\Bigg(  \alpha_j \le |\Grande_1^{\by,R,\ngrande,\npiccola}(0)| < \alpha_{j+1} \text{ and }
                              \sup_{\substack{|t-s|<\delta \\ 0\le s,t \le1}}            
                              |\WGrande_1^{\by,R,\ngrande,\npiccola}(t)-\WGrande_1^{\by,R,\ngrande,\npiccola}(s)| 
                              > \eps+j \Bigg) .                  
\end{align*}
We then use the following standard Brownian estimate:
\begin{lemma}\label{LemmeBrownianEstimate}
If $ W $ is a $d$-dimensional Brownian motion %on $(\Omega,\mathcal{F},P)$
then, for every $ \eps > 0 $ and every $ \delta \in ]0,1] $,
\begin{equation*}
    P\left( \sup_{\substack{ |t-s|<\delta,~ 0 \le s,t \le 1 } } ~ |W(t) - W(s)| ~ \ge \eps \right) 
    \quad \le~ 4\sqrt{5} ~ \frac{d}{\delta} ~ \exp\big(-\frac{\eps^2}{10\, d \,\delta}\big). 
\end{equation*}
\end{lemma}
The Brownian motion is independent of the starting point in equation~\eqref{eq:SDEnSpheres}, thus Lemma~\ref{LemmeBrownianEstimate} implies:
\begin{align*}
   & Q^\by_{R,\ngrande,\npiccola}\left( \Grande_1^{\by,R,\ngrande,\npiccola}(\cdot)\in\BadPath_{\text{Fast}}(\alpha,\delta,\eps) \right) \le 8\sqrt{5} \, \frac{d}{\delta} \,\sum_{j=0}^{+\infty} 
         \nu^\by_{R,\ngrande,\npiccola}\Bigg(  \alpha_j \le |\grande_1| < \alpha_{j+1} \Bigg)  
         \exp\left( -\frac{(\eps+j)^2}{10\, d \, \delta} \right).
\end{align*}
Then~\eqref{eq:nu_yRn} yields:
\begin{equation*}
    \nu^\by_{R,\ngrande,\npiccola}\Big( \alpha_j \le |\grande_1| < \alpha_{j+1} \Big)         
   \le \nu^\by_{R,\ngrande-1,\npiccola}(\D) ~ |B(0,\alpha_{j+1}) \setminus B(0,\alpha_j)|,
\end{equation*}
so that the above quantity is bounded by: 
\begin{equation*}
\begin{split}
   & \le 8\sqrt{5} ~\frac{d}{\delta} ~\exp(-\frac{\eps^2}{10\, d \, \delta}) \,
         \nu^\by_{R,\ngrande-1,\npiccola}(\D) ~
         \sum_{j=0}^{+\infty} e^{-\frac{2\eps j+j^2}{10\, d \,\delta}} ~ |B(0,\alpha_{j+1})| .
\end{split}
\end{equation*} 
The above sum is smaller than some constant $\Cgrande$ which only depends on the dimension $d$. Putting in $\Cgrande$ the factors depending only on the dimension, we obtain:
\begin{align*}
    &Q^\by_{R,\ngrande,\npiccola}\left( \Grande_1^{\by,R,\ngrande,\npiccola}(\cdot)\in\BadPath_{\text{Fast}}(\alpha,\delta,\eps) \right)
    \le \frac{\Cgrande}{\delta^{d+1}} ~ \exp(-\frac{\eps^2}{10\, d \,\delta}) 
        ~\nu^\by_{R,\ngrande-1,\npiccola}(\D) ~ \alpha^d. 
\end{align*}
We evaluate in the same way the particle term
$\disp Q^\by_{R,\ngrande,\npiccola}\left( \Piccola_1(\cdot)\in\BadPath_{\text{Fast}}(\alpha,\delta,\eps) \right)$.
The only difference is that the diffusion coefficient is equal to $\sigmaP$ instead of $1$.
We obtain again, for $\eps\le1$, $\delta\le1$, and $\alpha\ge1$, and with some constant $\Cpiccola$ which depends only on the dimension $d$:
\begin{align*}
    &Q^\by_{R,\ngrande,\npiccola}\left( \Piccola_1^{\by,R,\ngrande,\npiccola}(\cdot)\in\BadPath_{\text{Fast}}(\alpha,\delta,\eps) \right) \le \frac{\Cpiccola}{\delta^{d+1}} ~ \exp(-\frac{\eps^2}{10\, d \,\delta~\sigmaP^2}) 
        ~\nu^\by_{R,\ngrande,\npiccola-1}(\D) ~ \alpha^d. 
\end{align*}
Thanks to~\eqref{eq:exchangeability},
\begin{align*}
   &Q^\by_{R,\ngrande,\npiccola}\left( \BadPath'_{\text{Fast}}(\alpha,\delta,\eps) \right)\\
   & \le \frac{\alpha^d}{\delta^{d+1}} ~ 
         \exp(-\frac{\eps^2}{10\, d \, \delta~\max(1,\sigmaP^2)})
         \left( \ngrande~\Cgrande ~\nu^\by_{R,\ngrande-1,\npiccola}(\D) ~ 
                +\npiccola~\Cpiccola ~\nu^\by_{R,\ngrande,\npiccola-1}(\D) \right).
\end{align*}
We can now compute the probability of fast motion under the Poissonian mixture $Q^\by_R$:
\begin{align*}
   &Q^\by_R\left( \BadPath'_{\text{Fast}}(\alpha,\delta,\eps) \right)\\ 
   &\le \frac{\alpha^d}{\delta^{d+1}} ~ 
         \exp(-\frac{\eps^2}{10\, d \,\delta~\max(1,\sigmaP^2)})
         \frac{1}{\Zmixing^\by_R} \\
         &\quad \Bigg( \zgrande~\Cgrande
                 \underbrace{ \sum_{\ngrande=1}^{+\infty} \frac{\zgrande^{\ngrande-1}}{(\ngrande-1)!} \sum_{\npiccola=0}^{+\infty} \frac{\zpiccola^\npiccola}{\npiccola!}~\nu^\by_{R,\ngrande-1,\npiccola}(\D) }_{ =\Zmixing^\by_R } 
                 +\zpiccola~\Cpiccola \underbrace{ \sum_{\ngrande=0}^{+\infty} \frac{\zgrande^\ngrande}{\ngrande!} \sum_{\npiccola=1}^{+\infty} \frac{\zpiccola^{\npiccola-1}}{(\npiccola-1)!}~\nu^\by_{R,\ngrande,\npiccola-1}(\D) }_{ =\Zmixing^\by_R }
         \Bigg).
\end{align*}  
Choosing $C_{\text{Fast}} := \Cgrande~\zgrande + \Cpiccola~\zpiccola$
completes the proof of Lemma~\ref{prop:OmegaFast}.

\section*{Acknowledgments}
The authors would like to thank the editors and anonymous referee for their constructive comments, which
improved the quality of this paper.
MF acknowledges the support of the CDP C2EMPI, together
with the French State under the France-2030 programme, the University of Lille,
the Initiative of Excellence of the University of Lille, the European Metropolis of Lille for their funding and support of the R-CDP-24-004-C2EMPI project.

\bibliographystyle{alpha}
{\small\bibliography{References}}

\end{document}